[1994/12/01]
\documentclass{amsart}
\chardef\bslash=`\\ 

\usepackage{amssymb}

\usepackage[foot]{amsaddr}

\usepackage{color}
\usepackage{graphicx}

\usepackage{caption}
\usepackage{wrapfig}
\usepackage{lscape}
\usepackage{rotating}

\usepackage{overpic}

\usepackage{adjustbox}

\usepackage{hvfloat,lipsum}

\usepackage{caption}
\usepackage{subcaption}

\usepackage{appendix}

\usepackage{url}

\newtheorem{thm}{Theorem}[section]

\newtheorem{lem}[thm]{Lemma}

\newtheorem{conjecture}{Conjecture}

\theoremstyle{definition}

\newtheorem{rem}{Remark}[section]

\newtheorem*{ex}{Example}

\theoremstyle{remark}

\newcommand{\thmref}[1]{Theorem~\ref{#1}}

\newcommand{\secref}[1]{Section~\ref{#1}}

\newcommand{\remref}[1]{Remark~\ref{#1}}
\newcommand{\figref}[1]{Figure~\ref{#1}}

\definecolor{gray}{rgb}{0.5,0.5,0.5}

\newcommand{\obsgray}[1]{\textcolor{gray}{#1}}

\newcommand{\toinf}[1]{#1\rightarrow \infty}

\newcommand{\divides}{\mid}

\newcommand{\e}{\mathrm{e}}

\DeclareMathOperator{\li}{li}

\makeatletter
\def\imod#1{\allowbreak\mkern10mu({\operator@font mod}\,\,#1)}
\makeatother

\newcommand{\eval}[2][\right]{\relax
  \ifx#1\right\relax \left.\fi#2#1\rvert}

\title[The origin of $\textrm{li}(x)$ in the PNT]{The origin of the logarithmic integral in the prime number theorem}
\author{Kolbj{\o}rn Tunstr{\o}m}
\address{Dept. of Ecology and Evolutionary Biology\\ 
Princeton University\\
NJ, 08540, USA.}
\email{ktunstrom@gmail.com}

\begin{document}

%

\begin{abstract} 

We establish why $\li(x)$ outperforms $x/\log x$ as an estimate for the prime counting function $\pi(x)$. The result follows from subdividing the natural numbers into the intervals \mbox{{$s_{k} :=\{p_k^2, \dots,p_{k+1}^2-1\}$}}, $k\geq1$, each being
fully sieved by the $k$ first primes $\{p_1, \dots,p_{k}\}$. Denoting the number of primes in $s_k$ by $\pi_k$, we show that $\pi_k \sim |s_k| / {\log p_{k+1}^2}$ and that $\pi(x) \sim \li(x)$ originates as a continuum approximation of the sum $\sum_k \pi_k$. In contrast, $\pi(x) \sim x / \log x$ stems from sieving
repeatedly in regions already completed 
---explaining why $x / \log x$ underestimates $\pi(x)$. 
The explanatory potential arising from defining $s_k$ appears promising, 
evidenced in the last section where we outline further research. 
\end{abstract}



 \maketitle




\section{Introduction}
\label{sec:intro}
The prime number theorem states that as $x \rightarrow \infty$, the number of primes below $x$, denoted $\pi(x)$, can be approximated by either relations
\begin{align}
	\pi(x) \sim \frac{x}{\log x} \quad \textrm{or} \quad \pi(x) \sim \li(x), 
\label{eq:pnt}	
\end{align}
where $\li(x)$ is the logarithmic integral defined by $\li(x):= \int_2^x \frac{dt}{\log t}$. While the two estimates in \eqref{eq:pnt} are equivalent---easily proved by a series expansion of $\li(x)$, where $x/\log x$ appears as the leading term---they differ in performance. As is well known, established by proof already in 1899 by de la Vall\'{e}e-Poussin \cite{Poussin:1899}, $\li(x)$ is a superior guess of the number of primes up to $x$ compared to $x/\log x$. Nonetheless, the exact reason \textit{why} $\li(x)$ outcompetes $x/\log x$ as an estimator for $\pi(x)$ has remained unresolved. The absence of an explanation is evident in various literature surveys \cite{DIAMOND:ug, Fine:2010vu,Goldfeld:2004wh, Goldston:2007wk}, and perhaps most clearly stated by Goldston \cite{Goldston:2007wk}: \textit{'The extraordinarily good fit between $\pi(x)$ and $\li(x)$, far better than the first approximation $x /\log x$, has been the subject of intensive but largely unsuccessful investigation for the last one hundred years.'}

In this paper,  we introduce a slight shift of perspective on the sieve of Eratosthenes. Rather than focusing directly on the number of primes below a given number $x$, we examine what the effect is of sieving by the $k$th prime $p_k$. This eventually inspires a subdivision of the natural numbers in terms of the specific intervals $s_k:=\{p_{k}^2, \dots p_{k+1}^2-1\}$, where each $s_k$ has the essential property that it is sieved entirely by the $k$ first primes. By writing the number of primes in $s_k$ as $\pi_k$, we have from the prime number theorem that $\pi_k \sim |s_k| / {\log p_{k+1}^2}$. Summing these estimates up to the desired value of $x$ gives us $\pi(x)$, and it is easy to see that $\li(x)$ arises as a continuum approximation of this sum. $x/\log x$, on the other hand, derives from sieving the whole interval below $x$ by every prime $p\leq\sqrt{x}$, effectively overcounting the number of composites and underestimating ditto primes. As a consequence of these insights, we are able to explain the relation between $x/\log x$ and $\li(x)$ in the prime number theorem.

The remaining parts of the paper consist of three sections. In the first, we expand in a slower pace the content of the previous paragraph, motivating the subdivision of the natural numbers from the sieve of Eratosthenes, and stating the definition of the intervals $s_k$. From the definition we establish probabilistic estimates of the number of primes within each interval, and in turn, this naturally steers us towards what we could call a probabilistic prime number theorem. The insight we acquire from the probabilistic perspective motivates an alternative formulation of the prime number theorem---where the logarithmic integral now originates as a continuum approximation of the sum taken over the individual interval estimates. The second section contains a more compact overview of results and detailed proofs. 
In the third and last section we explore some of the implications following in the wake of defining $s_k$.  Most notably, we discuss how our results relate to the studies of primes in short intervals and the error term in the prime number theorem.

\section{Deriving $\li(x)$ from a subdivision of the natural numbers}
\label{sec:subdivide}

In 1849, the German astronomer Encke wrote a letter to his former academic advisor Gauss, in which he shared his thoughts on the frequency of primes. In his reply,  Gauss explained that his first inquiries into the topic dated as far back as 1792 or 1793, when he had received the Lambert supplements to the logarithmic tables, containing lists of all primes up to 1 million---and also that he had kept on adding to that million: Gauss wrote to Encke that he 'frequently spent an idle quarter of an hour to count a chiliad\footnote{Gauss counted the primes within \textit{chiliads}; intervals of 1000 consecutive integers.} here and there', which allowed him to calculate the average number of primes over short intervals, eventually conjecturing that\footnote{The historical account is found in \cite[p. 174]{Havil:2003}}
\begin{align*}
	\pi(x) \sim \li(x).	
\end{align*}

As we will see, Gauss's approach of probing the distribution of primes across short intervals is key to understanding the emergence of the logarithmic integral. Essentially we will apply the same strategy, with the critical distinction that our variant of Gauss's chiliads will be a set of precisely stated intervals that subdivides the natural numbers. Their definition is best motivated by a recapitulation of the sieve of Eratosthenes, which is a straightforward sieving procedure for finding all primes up to a given number $x$ and a common starting point for understanding the distribution of primes.

Fairly typical of what is found in textbooks or other literature---see e.g. \cite{Greaves:2001, Goldston:2007wk, Granville:2010cq}---we can describe the sieve of Eratosthenes as follows: First we remove all multiples of the first prime $p_1=2$ up to $x$. The next remaining number is the second prime $p_2=3$, so now we remove all its surviving multiples up to $x$. Thereafter, we repeat the process until no more composites can be removed by further sieving. This happens when we have sieved by the primes $p\leq \sqrt{x}$, as any composite with prime factors all larger than $\sqrt{x}$ must necessarily exceed $x$. What's more, the sieve of Eratosthenes lets us construct a probabilistic estimate of the number of primes up to $x$ by the following reasoning: The probability that a randomly picked number less than or equal to x is prime equals the probability that the said number is not divisible by any of the primes $p\leq \sqrt{x}$. Treating these events as independent, we arrive at the probability given by the Euler product 
\begin{align*}
	\prod_{p\leq \sqrt{x}} \left(1- \frac{1}{p}\right),
\end{align*}
and the expected number of primes up to $x$ is therefore 
\begin{align}
	\tilde \pi(x) = x \cdot 	\prod_{p\leq \sqrt{x}} \left(1- \frac{1}{p}\right), 
\label{eq:sieve_estimate}
\end{align}
where $\tilde \pi(x)$ is introduced as a probabilistic prime counting function. 

What now if we wanted to guess the number of primes up to some $y>x$, instead of $x$? Applying the same arguments, we could produce an estimate simply by replacing $x$ by $y$ in \eqref{eq:sieve_estimate}. However, by doing so, we would not factor in that we already have sieved the interval up to $x$. Consequently, the estimate in \eqref{eq:sieve_estimate} fails to take into account that every sieve step leaves behind it a completely sieved interval. 


To remedy this shortcoming, let us dim the light on the upper limit $x$ and shift focus towards what happens during one step of the sieve of Eratosthenes. For this purpose, let $\mathcal P_k:=\{p_1,\dots, p_k\}$ be the set of the $k$ first primes, and assume we have already sieved the natural numbers by all primes up to and including $p_{k-1}$. Necessarily, $p_{k}^2$ is the next composite that is not sieved, and therefore, when proceeding with step $k$, the sieving has no effect on numbers smaller than $p_{k}^2$---we are only removing elements from $p_{k}^2$ and upwards, and thereby locating all primes between $p_k^2$ and $p_{k+1}^2$. Evidently, for any $k\geq 1$, the interval \mbox{$\{p_k^2, \dots, p_{k+1}^2-1\}$} can be sieved apart from any other interval by the $k$ first primes. That being the case, what transpires as our version of Gauss's chiliads is the set of intervals defined by 
\begin{align*}
	\mbox{$s_k:=\{p_k^2, \dots, p_{k+1}^2-1\}$}. 
\end{align*}
To summarize this insight:

\begin{rem}
\label{remark1}
A logical consequence of the sieve of Eratosthenes is that the natural numbers can be split into the set of intervals $s_k$, $k\geq1$, where $s_k$ has the vital property that any one of  its elements is either divisible by some $p \in \mathcal P_k$ or else is a prime $p \notin \mathcal P_k$. Therefore, each $s_k$ can be sieved independently by the $k$ first primes.
\end{rem}

Now, let us return to the probabilistic prime counting function. How will our knowledge of $s_k$ and the fact that it can be sieved by the $k$ first primes affect $\tilde \pi(x)$? If we let the length of $s_k$ be given by $l_k := |s_k| = p_{k+1}^2-p_{k}^2$, and introduce $\tilde \pi_k$ to be the expected number of primes in $s_k$, it follows from the arguments leading to \eqref{eq:sieve_estimate} that
\begin{align}
	\tilde \pi_k = l_k \cdot \prod_{p \in \mathcal P_k} \left(1- \frac{1}{p}\right).
\label{eq:tildepi_k1}
\end{align}
If we further assume $k$ to be the integer such that $p_k^2 \leq x < p_{k+1}^2$, then it follows that our improved estimate of $\tilde \pi(x)$---now incorporating that each interval $s_k$ is fully sieved by the $k$ first primes---is given by
\begin{align}
	\tilde \pi(x) =  \sum_{j=1}^{k-1} \tilde \pi_j
			+ \frac{x-p_k^2}{l_k} \, \tilde \pi_k. 
\label{eq:sum_sieve_estimate}
\end{align}
	
A natural next step is to apply Merten's product theorem to \eqref{eq:sieve_estimate} and \eqref{eq:sum_sieve_estimate},
which leads to the probabilistic estimates
\begin{align}
	\tilde \pi(x) \sim  2 \textrm{e}^{-\gamma} \frac{x}{\log x} \quad \textrm{and} \quad \tilde \pi(x) \sim  2 \textrm{e}^{-\gamma} \li (x).
\label{eq:MPNT0}
\end{align}
The left estimate here corresponds to \eqref{eq:sieve_estimate} and is previously stated in the literature, see e.g. \cite{Granville:1995cm}. The right---and more precise estimate---corresponds to \eqref{eq:sum_sieve_estimate}, and is to the best of our knowledge an original result. The intermediate details necessary to arrive at this estimate are found in the proof of \thmref{thm:MPNT2}. Most importantly, we have arrived at two probabilistic estimates that each have their counterpart in the prime number theorem \eqref{eq:pnt}; except for a constant factor $2 \textrm{e}^{-\gamma}$, \eqref{eq:MPNT0} is identical to \eqref{eq:pnt}. 

As we explain in more detail in \secref{sec:notation}, the probabilistic estimates in \eqref{eq:MPNT0} can be understood as arising from finite sample spaces of possible prime positions within the specific intervals $s_k$. The positions of the actual primes naturally lie in these sample spaces and therefore the arguments we have laid out above also apply to the primes. In conclusion then, the reason why $\li (x)$ is a better estimate for $\pi(x)$ than $x/\log x$ is that it takes into account that each sieve step $k$ completes the sieving of the interval $s_k$, while $x/ \log x$ does not. Accordingly, this suggests that $\li (x)$ is the appropriate function to use for estimating the number of primes up to $x$.

Rather than \eqref{eq:MPNT0}, it appears natural now to write
\begin{align}
	\tilde \pi_k \sim  2 \textrm{e}^{-\gamma} \frac{l_k}{\log p_{k+1}^2} \quad \textrm{and} \quad \tilde \pi(x) \sim  2 \textrm{e}^{-\gamma} \li (x),
\label{eq:MPNT00}
\end{align}
where the left term---derived from \eqref{eq:tildepi_k1} via Merten's theorem (see \thmref{thm:MPNT1})---provides the expected number of primes within the intervals $s_k$, and the right term---obtained as a continuous approximation of $\sum_k  \tilde \pi_k$ (see \thmref{thm:MPNT2})---describes the expected number of primes up to $x$. Furthermore, the expression for $\tilde \pi_k$ in \eqref{eq:MPNT00} reflects the underlying sieving process, in that sieving by the $k$ first primes removes all composites up to $p_{k+1}^2$.

Following through to completion, \eqref{eq:MPNT00} points towards the alternative formulation of the prime number theorem given by
\begin{align}
	\pi_k \sim  \frac{l_k}{\log p_{k+1}^2} \quad \textrm{and} \quad  \pi(x) \sim \li (x),
\label{eq:PNT0}
\end{align}
where $\pi_k$ is the number of primes in $s_k$ (see \thmref{thm:PNT1}), and the natural interpretation of the logarithmic integral is that it originates from a continuous approximation of $\sum_k \pi_k$  (see \thmref{thm:PNT2}).

\section{Results and proofs}
Let $ \mathcal P_k:=\{p_1, \dots, p_k\}$ be the set consisting of the $k$ first primes. Also, let the interval $s_k$ be defined by \mbox{$s_k:=\{p_k^2, \dots, p_{k+1}^2-1\}$},
and denote the length of the interval $s_k$ by $l_k := |s_k| = p_{k+1}^2 - p_{k}^2$. Then we have the following results:

\begin{lem} 
\label{lem:lem1}
Let $\tilde \pi_k$ denote the expected number of primes in the interval $s_k$. Then 
\begin{align}
	\tilde \pi_k = l_k \cdot \prod_{p \in \mathcal P_k} \left(1- \frac{1}{p}\right).
\label{eq:tildepi_k2}
\end{align}
\end{lem}

\begin{proof} 
Each number in $s_k$  is either divisible by some $p \in \mathcal P_k$ or else is a prime $p \notin \mathcal P_k$. The probability of a number in $s_k$ being prime is therefore equivalent to the probability of the number not being divisible by any $p\in \mathcal P_k$. Assuming these events to be independent, the desired probability is given by the Euler product  
\begin{align*}
	\prod_{p \in \mathcal P_k} \left(1- \frac{1}{p}\right),
\end{align*}
and the expected number of primes in $s_k$ is found by multiplying with the length $l_k$ of the sequence, resulting in \eqref{eq:tildepi_k2}.
\end{proof}

\begin{lem} 
\label{lem:lem2} 
Let $\tilde \pi(x)$ denote the expected number of primes less than or equal to $x$ and let $k$ be the integer such that $p_k^2 \leq x < p_{k+1}^2$. Then
\begin{align}
	\tilde \pi(x) =  \sum_{j=1}^{k-1} \tilde \pi_j
			+ \frac{x-p_k^2}{l_k} \, \tilde \pi_k. 
\label{eq:tildepi_x}
\end{align}
\end{lem}

\begin{proof} 
The result follows by splitting $\tilde \pi(x)$ into a sum of two parts,
\begin{align}
	\tilde \pi(x) = \tilde \pi(p_{k}^2) +  [\tilde \pi(x) - \tilde \pi(p_{k}^2)]. 
\label{eq:tildepi_split}	
\end{align}
The first term on the right can be split further into a sum over the intervals $s_j$, \mbox{{$1\leq j <k$}}, 
\begin{align*}
	\tilde \pi(p_{k}^2) = 
	\sum_{j=1}^{k-1} \left[ \tilde \pi(p_{j+1}^2) - \tilde \pi(p_{j}^2) \right] = \sum_{j=1}^{k-1} \tilde \pi_j. 
\end{align*}
The second term on the right in \eqref{eq:tildepi_split} represents the mean number of primes in the interval $[p_k^2, x]$. Since this interval is contained within $s_k$ it follows that 
\begin{align*}
	\tilde \pi(x) - \tilde \pi(p_{k}^2) 
	&= \left(x-p_k^2\right) \cdot \prod_{p \in \mathcal P_k} \left(1- \frac{1}{p}\right)\\
	&= \frac{x-p_k^2}{l_k}  \left[ l_k \cdot \prod_{p \in \mathcal P_k} \left(1- \frac{1}{p}\right) \right]\nonumber\\
	&= \frac{x-p_k^2}{l_k}  \, \tilde \pi_k. \nonumber
\end{align*}
Combining these two results produces \eqref{eq:tildepi_x}.
\end{proof}
\begin{thm}
\label{thm:MPNT1} 
Let $\tilde \pi_k$ denote the expected number of primes in the interval $s_k$. When $k \rightarrow \infty$ we have that 
\begin{align}
	\tilde \pi_k \sim  2 \textrm{e}^{-\gamma} \frac{l_k}{\log p_{k+1}^2}.
\label{eq:tildepi_k3}
\end{align}
\end{thm}

\begin{proof} Assume that $x$ is a real number within the interval $s_k$, so that $p_k^2 \leq x < p_{k+1}^2$, and consider the situation when $k \rightarrow \infty$. Then we can state Merten's product theorem 
\mbox{\cite{Villarino:2005to}} 
as
\begin{align*}
	\prod_{p \in \mathcal  P_k} (1-\frac{1}{p}) = \e^{-\gamma+\delta}\frac{1}{ \log \sqrt{x}} = 2 \e^{-\gamma+\delta}\frac{1}{ \log x}, 
\end{align*}
where $\gamma$ is the Euler--Mascheroni constant and $\delta$ is a measure of the uncertainty of the approximation, satisfying
\begin{align}
	|\delta | < \frac{4}{\log (\sqrt{x}+1)} + \frac{2}{\sqrt{x} \log \sqrt{x}} + \frac{1}{2 \sqrt{x}}.  
\label{eq:mertenerror}
\end{align}
It now follows immediately from \eqref{eq:tildepi_k2} that 
\begin{align}
	\tilde \pi_k =  2 \textrm{e}^{-\gamma+\delta} \frac{l_k}{\log x}
\label{eq:tildepi_k_x}	
\end{align}
for any $x$ such that $p_k^2\leq x < p_{k+1}^2$. Additionally, \eqref{eq:mertenerror} implies that in order to minimize the error we should choose $x$ as large as possible---that is, $x=p_{k+1}^2-\epsilon$, where $\epsilon>0$ is infinitesimal---by which we obtain
\begin{align*}
	\tilde \pi_k
	\sim 2 \e^{-\gamma} \frac{l_k}{ \log  p_{k+1}^2}. 
\end{align*}

Note that even without access to the error term \eqref{eq:mertenerror}, it would be natural to argue heuristically for the choice of $x=p_{k+1}^2$ by the following reasoning:
The estimate of $\tilde \pi_k$ should reflect the characteristic property of the sieving process, that sieving by the $k$ first primes removes all composites up to and including $p_{k+1}^2-1$. 
\end{proof}

\begin{thm}
\label{thm:MPNT2}
Let $\tilde \pi(x)$ denote the expected number of primes less than or equal to $x$. When $x \rightarrow \infty$ we have that 
\begin{align*}
	\tilde \pi(x) \sim  2 \e^{-\gamma} \li (x).
\end{align*}
\end{thm}

\begin{proof} We start by defining $\li_k$ as the logarithmic integral over the interval $s_k$,
\begin{align*}
	\li_k:= \li(p_{k+1}^2) - \li(p_{k}^2) = \int_{p_k^2}^{p_{k+1}^2} \frac{dt}{\log t}. 
\end{align*}
As $\log a < \log b$ whenever $a<b$ for any real numbers $a, b>0$, $\li_k$ satisfies
\begin{align}
	\frac{l_k}{\log {p_{k+1}^2} } < \li_k < \frac{l_k}{\log {p_{k}^2} }, 
\label{eq:li_inequal}
\end{align}
and since---from \eqref{eq:tildepi_k_x}---the following relations hold when $k \rightarrow \infty$,
\begin{align*}
	\tilde \pi_k \sim 2 \e^{-\gamma} \frac{l_k}{\log {p_{k+1}^2} } \quad \textrm{and} \quad \tilde \pi_k \sim 2 \e^{-\gamma} \frac{l_k}{\log {p_{k}^2} },
\end{align*}
we have that 
\begin{align}
	\tilde \pi_k \sim 2 \e^{-\gamma} \li_k.
\label{eq:li_k_mertens}
\end{align}

Now, let $k$ be the integer such that $p_k^2 \leq x < p_{k+1}^2$. Then it follows from \eqref{eq:tildepi_x} and \eqref{eq:li_k_mertens} that
\begin{align*}
	\tilde \pi(x) 
	&=  \sum_{j=1}^{k-1} \tilde \pi_j + \frac{x-p_k^2}{l_k} \, \tilde \pi_k \nonumber\\
	&\sim 2 \e^{-\gamma} \sum_{j=1}^{k-1} \li_j + \frac{x-p_k^2}{l_k} \, 2 \e^{-\gamma} \li_k \\
	&= 2 \e^{-\gamma} \li(x)\nonumber.
\end{align*}
Strictly, the last expression should be $2 \e^{-\gamma} \li(x)-2$, since $\li(x)$ has lower integration limit $x=2$, thereby including the 2 first primes 2 and 3. But that difference is insignificant as $\toinf{x}$.
\end{proof}

\begin{rem}
\label{rem:relerror} 
Note that we can easily obtain an upper bound on the relative error 
\begin{align}
	\eta_k 
	&:= \frac{l_k/\log {p_{k}^2} - l_k/\log {p_{k+1}^2} }{l_k/\log {p_{k+1}^2}} 
	=\frac{\log {p_{k+1}^2}}{\log {p_{k}^2}} - 1
\label{eq:eta1}
\end{align}
by applying Bertrand's postulate. This states that for any integer $n>1$ we will always find a prime $p$ such that $n < p < 2 n$. Choosing $n$ to be $p_k$, we see that $p_k < p_{k+1} < 2 p_{k}$ and therefore that $p_{k+1}^2 < 4 p_{k}^2$. Then it follows from \eqref{eq:eta1} that 
\begin{align}
	\eta_k 
	=\frac{\log {p_{k+1}^2}}{\log {p_{k}^2}} - 1 
	<\frac{\log {4p_{k}^2}}{\log {p_{k}^2}} - 1 
	=\frac{\log {4}}{\log {p_{k}^2}}.
\label{eq:eta2}
\end{align}
So, as $k\rightarrow \infty$, the relative error approaches 0. Naturally, the actual error approaches 0 much faster than what we found here, as Bertrand's postulate is a very crude approximation to the average distance between primes. Applying the prime number theorem, the average gap length around $p_k$ is $\log p_k$, and hence $p_{k+1}\sim p_k + \log p_k$. The expected relative error should therefore satisfy
\begin{align*}
	\eta_k 
	&\approx \frac{\log ({p_{k} + \log p_k}) }{\log {p_{k}} } - 1. 
\end{align*}
Necessarily, since $\li_k$ is bounded as shown in \eqref{eq:li_inequal}, the relative error of any numerical approximation to $\li_k$ must obey \eqref{eq:eta2} and approach 0 as $k
\rightarrow \infty$.
\end{rem}

\begin{thm} 
\label{thm:PNT1}
Let $\pi_k$ denote the number of primes in the interval $s_k$. When $k \rightarrow \infty$ we have that 
\begin{align}
	\pi_k \sim  \frac{l_k}{\log p_{k+1}^2}.
\label{eq:PNT1}	
\end{align}
\end{thm}

\begin{proof}
The truth of the statement follows directly from the original prime number theorem, since
\begin{align*}
	\pi_k 
	&= \pi(p_{k+1}^2)-\pi(p_{k}^2) \sim \li(p_{k+1}^2)-\li(p_{k}^2) = \li_k \\
	&= \int_{p_{k}^2}^{p_{k+1}^2} \frac{dt}{\log t} \approx \frac{l_k}{\log p_{k+1}^2}.\nonumber
\end{align*}
The choice of the integral approximation 
\begin{align*}
	\frac{l_k}{\log p_{k+1}^2}
\end{align*}
rather than e.g. 
\begin{align*}
	\frac{l_k}{\log p_{k}^2} \quad \textrm{or} \quad \frac{l_k}{\log \frac{1}{2}\left(p_{k}^2 +p_{k+1}^2 \right)}
\end{align*}
is motivated by the proof of \thmref{thm:MPNT1}, while the accuracy of the approximation is discussed in \remref{rem:relerror} above.
\end{proof}

\begin{thm}
\label{thm:PNT2}
Let $\pi(x)$ denote the number of primes less than or equal to $x$. When $x \rightarrow \infty$ we have that 
\begin{align}
	\pi(x) \sim  \li (x).
\label{eq:PNT2}
\end{align}
\end{thm}

\begin{proof}
This is already stated by the original prime number theorem, but here---following the proof of \thmref{thm:MPNT2}---we suggest that \eqref{eq:PNT2} should be interpreted as the continuous approximation of $\sum_k \pi_k$, where $\pi_k$ is defined as in \eqref{eq:PNT1}.
\end{proof}

\section{Possible directions}

Initial investigations---theoretical and numerical---suggest that the subdivision of the natural numbers in terms of $s_k$ provides a rich a structure from which to probe deeper in our understanding of the prime numbers; or even, reinterpreting old results in terms of the intervals $s_k$. In the following, we briefly sketch out a few directions where we believe the insights of this paper could have impact, keeping in mind that there are also several additional routes to explore than those mentioned here.
The format is mixed, in the sense that what we present range across rigorous results, heuristics and experimental insights. As such, this section is best read as a communication of ideas possibly worthy of exploration. 
%
%
%

Accompanying this section is also a set of figures included as an appendix. The Mathematica software were used for generating and analyzing data, and creating the final figures. 
%

\subsection{Notation}
\label{sec:notation}
%
%
%
%
%
%

Here we introduce notation that will be used in later sections. First, we construct an arithmetic function that describes which integers $n$ are divisible by a given prime $p_k$:
\begin{align*}
\rho_{k}(n):=\begin{cases}
    p_k & \text{if $p_k \divides n$,}\\
    1 & \text{otherwise}.
  \end{cases}
\end{align*}
Now, let us apply this definition to construct a second arithmetic function that locates all $n$ coprime to $p_k\#:=\prod_{p\in \mathcal P_k}p$:
\begin{align*}
	  R_{k}(n) :=   \prod_{1\leq i \leq k} \rho_{i}(n).
\end{align*}
We see that $(n,p_k\#)=1$ whenever $R_{k}(n)=1$. Also, both $\rho_{k}(n)$ and $R_{k}(n)$ are periodic, satisfying for  any integer $m$ the equalities
\begin{align*}
	\rho_{k}(n + m p_k) = \rho_{k}(n) \quad \textrm{and} \quad 	R_k(n+ m p_k \#) = R_k(n).
\end{align*}
This notation is also useful for visualizing the positions of primes in $s_k$, as exemplified here by $s_3$:\\
{\scriptsize
$$
\begin{tabular}{c|ccccccccccccccc|c}
\obsgray{$s_3$} & 25 & 26 & 27 & 28 & 29 & 30 & 31 & 32 & 33 & 34 & 35 & 36 & 37 & $\cdots$ & 48 &  \\
\hline  \\ [-3mm]
\obsgray{$\rho_1$}  & 1 & $p_1$ & 1 & $p_1$ & 1 & $p_1$ & 1 & $p_1$ & 1 & $p_1$ & 1 & $p_1$ & 1 &$\cdots$ & $p_1$ & $\cdots$\\
\obsgray{$\rho_2$}  &1 & 1 & $p_2$ & 1 & 1 & $p_2$ & 1 & 1 & $p_2$ & 1 & 1 & $p_2$ & 1 &$\cdots$ & $p_2$ & $\cdots$\\
\obsgray{$\rho_3$}  &$p_3$ & 1 & 1 & 1 & 1 & $p_3$ & 1 & 1 & 1 & 1 & $p_3$ & 1 & 1 &$\cdots$ & 1 & $\cdots$\\
\hline  \\ [-3mm]
\obsgray{$R_3$}  & $p_3$ & $p_1$ & $p_2$ & $p_1$ & 1 & $p_1 p_2 p_3$ & 1 & $p_1$ & $p_2$ & $p_1$ & $p_3$ & $p_1p_2$ & 1 &$\cdots$ & $p_1p_2$  & \dots 
\end{tabular}
$$}

Consider next an arbitrary interval $A$. The number of coprimes to $p_k\#$ in $A$ is equivalent to the count of 1s in $R_k(n)$ across $A$, in sieve notation denoted by
%
%
%
%
%
\begin{align*}
	S(A,p_k\#) := |\{n: n\in A, R_k(n)=1 \}|.
\end{align*}
It follows from the periodicity of $R_k(n)$ that also $S(A,p_k\#)$ is periodic; if $A^n$ denotes $A$ left-shifted $n$ times, the equality
\begin{align*}
	S(A^{mp_k\#},p_k\#) = S(A,p_k\#)
\end{align*}
holds for any integer m. We can express $S(A,p_k\#)$ in terms of the Legendre identity as
\begin{align*}
	S(A,p_k\#) = 
	\lfloor |A| \rceil
	- \sum_{q\in \mathcal P_k} \left\lfloor \frac{|A|}{q} \right\rceil
	+ \sum_{\substack{{q_1, q_2 \in \mathcal P_k} \\ {q_1 < q_2}}}   \left\lfloor \frac{|A|}{q_1 q_2} \right\rceil 
	- \dots 
	\pm \sum_{\substack{{q_1, \dots, q_k \in \mathcal P_k} \\ {q_1 < \dots < q_k} }}   \left\lfloor \frac{|A|}{q_1 \dots q_k} \right\rceil,
\end{align*}
where $\lfloor x \rceil$ is defined to mean that $x$ take either values $\lfloor x \rfloor$ or $\lceil x \rceil$. 
In the average case, the Legendre identity has the simpler form
\begin{align*}
	E[S(A,p_k\#)]   = |A| \cdot \prod_{p \in \mathcal P_k} \left(1- \frac{1}{p}\right).
\end{align*}

Assume now the interval $s_k$---the focal object throughout this paper.
%
By the sieve of Eratosthenes we know that in $s_k$, the primes are the only numbers coprime to $p_k\#$, so necessarily,
\begin{align*}
	\pi_k = S(s_k,p_k\#).
\end{align*}
%
In later sections, however, we approach the distribution of primes in $s_k$ probabilistically, 
in the sense that we consider
all possible arrangements of coprimes to $p_k\#$ within an arbitrary interval of length $l_k$.
Because of the periodicity of $R_k(n)$, there are only $p_k\#$ different arrangements, 
which we derive
 in terms of a shifted 
%
version of $s_k$,
\begin{align*}
	s^j_k := \{p_k^2+j,\dots, p_{k+1}^2-1+j\}, \quad 0\leq j < p_k\#.
\end{align*}
%
%
In this notation, the probabilistic prime counting function $\tilde \pi_k$ takes the form
\begin{align*}
	\tilde \pi_k =  \frac{1}{p_k\#} \sum_{j=0}^{p_k\#-1}  S(s_k^j,p_k\#),
\end{align*}
while $\pi_k$ coincides with the case $j=0$,
%
\begin{align*}
	\pi_k = S(s_k^0,p_k\#).
\end{align*}
\subsection{Primes in short intervals}
There is a rich literature on primes in short intervals---see e.g. \cite{Yildirim:1999ws, Yildirim:2009vv, Granville:2010cq}---related to the question for which functions $\Phi(x)$, as $x\rightarrow \infty$, 
\begin{align}
	\pi(x + \Phi(x))-\pi(x) \sim \frac{\Phi(x)}{\log x}.
\label{eq:short}
\end{align}
With respect to this matter, Selberg proved in 1943 \cite{Selberg:1943}---assuming the Riemann hypothesis---that \eqref{eq:short} holds for almost all $x$,  required $\Phi(x)/ (\log x)^2 \rightarrow \infty$ as $x\rightarrow \infty$. It was long thought that this would be true for all $x$ \cite{Granville:2010cq}, but in 1985, Maier \cite{Maier:1985ww} published a proof demonstrating that there are infinitely many exceptions to Selberg's result, even for functions $\Phi(x)$ growing much faster than  $(\log x)^2$. 

What Maier specifically proved---here in the formulation of Granville \cite{Granville:2010cq}---was that for any constant $\lambda>2$, there exists a constant $\delta_{\lambda} > 0$ such that we find arbitrarily large integers $x$ and $X$ for which
\begin{align}
\label{eq:Maier}
	\pi\left(x + (\log x)^{\lambda}\right) - \pi(x) &\geq (1+\delta_{\lambda})(\log x)^{\lambda-1},  \quad \textrm{and}\\
	\pi\left(X + (\log X)^{\lambda}\right) - \pi(X) &\leq (1-\delta_{\lambda})(\log X)^{\lambda-1}. \nonumber
\end{align}

How do our findings connect to these results? First of all, we recognize that the estimate in \eqref{eq:tildepi_k_x} of the number of primes in the interval $s_k$ is literally of the exact same form as \eqref{eq:short}. Choosing $x=p_k^2$ and $\Phi(x)=l_k$, and assuming $\toinf{x}$, we have that
\begin{align*}
	\pi(p_k^2 + l_k)-\pi(p_k^2) = \pi_k \sim  \frac{l_k}{\log p_k^2},
\end{align*}
and it immediately follows that, as $k \rightarrow \infty$, all intervals $s_k$ satisfy  \eqref{eq:short}. 

Now, if we substitute $x$ for $p_{k+1}^2$ and $g$ for $g_k$ in $l_k=2 p_{k+1} g_k - g_k^2$, we find that each $l_k$ lies on one of the curves given by the functions $l_g(x)=2 \sqrt{x} g - g^2$, where $g=2,4,\dots$. Obviously, all functions $l_g(x)$ eventually outgrow any function $(\log x)^{\lambda}$---no matter what value of $\lambda$ we start with---which demonstrate that the intervals $s_k$ fulfill the requirement in Selberg's theorem  that $\Phi(x)/ (\log x)^2 \rightarrow \infty$ as $x\rightarrow \infty$. 

Another consequence is that what we could call Maier's interval, $[x, x+(\log x)^{\lambda}]$, will always be infinitesimal compared to $s_k$ as $k$ moves towards infinity. We can therefore place Maier's interval within parts of the intervals $s_k$ where the fluctuations in prime density deviates from the density predicted by the prime number by a multiplicative constant $\delta_{\lambda}$. The example below should give a clear demonstration of this interpretation. 

A final note is that in Maier's theorem, $\Phi(x)$ is restricted to grow logarithmically, while that is not the case in Selberg's theorem. In some sense this suggests we can view the two theorems as describing the behavior of the distribution of the primes on different length scales. Of course---strictly---Selberg's theorem also applies on the scale of Maier's interval. This can even be understood from the fact that Mayer's theorem guarantees that the density of primes in the interval $[x, x+ \Phi(x)]$ crosses the mean value predicted by the prime number theorem infinitely many times as $\toinf{x}$. An alternative view is therefore that Selberg's theorem is picking out intervals near these crossing points.

\begin{ex}

Let us choose $\lambda = 3$. Then, according to \eqref{eq:Maier}, we should be able to find a $\delta_3$ such that the inequalities
\begin{align}
	\pi\left(x + (\log x)^3 \right) - \pi(x) &\geq (1+\delta_3)(\log x)^{2}  \quad \textrm{and} \label{eq:maierex}\\
	\pi\left(X + (\log X)^{3}\right) - \pi(X) &\leq (1-\delta_3)(\log X)^{2} \nonumber
\end{align}
hold for arbitrarily large $x$ and $X$. While arbitrarily large is beyond our reach, we will examine a few selected intervals---$s_{500}$, $s_{750}$, and $s_{1000}$---such that $\Phi(x)$ is significantly smaller than $l_k$ and fluctuations in prime density within the intervals $s_k$ become apparent. 

The following table provides an overview of how $\Phi(x)$ compares to $l_k$ for the chosen intervals. Note that as $\Phi(x)$ is almost constant across an interval $s_k$ for large $k$, so is the ratio $\Phi(x)/l_{k}$. The values of $\Phi(x)$ are rounded to nearest integer and those of $\Phi(x)/l_{k}$ are presented with two significant digits.
\vspace{3mm}
\begin{center}
  \begin{tabular}{ c | c | c | c | c | c}
    \hline
    k 		& $g_k$	& $l_k$	& $\Phi(p_{k}^2)$ 	& $\Phi(p_{k+1}^2)$ & $\Phi(p_{k}^2)/l_{k}$ 	\\ \hline
    500 	& 10		& 71250 		& 4380 			& 4384 		& 0.061 				\\ \hline
    750 	& 8		& 91152 		& 5172 			& 5175 		& 0.057 				\\ \hline
    1000 	& 8		& 126768 		& 5787 			& 5789 		& 0.046 				\\
    \hline
  \end{tabular}
\end{center}
\vspace{3mm}
For a broader comparison, we can observe how $\Phi(x)$ compares to $l_g(x)$ across a range of values for $x$. This is done in \figref{fig:interval_lengths}, where we have plotted the values of $l_k$ together with $l_g(x)=2 \sqrt{x} g - g^2$ and $\Phi(x)=(\log x)^3$. When $x$ is within the shaded region, $\Phi(x)$ is larger or equal to some $l_k$, in particular all those $l_k$ that lie on the curve $l_2(x)$. But put $x$ beyond shaded the region, and $\Phi(x)$ is smaller than $l_g(x)$ for all values of $g$, and consequently also smaller than any interval length $l_k$. The behavior seen here is general; a different choice of lambda would stretch out the shaded region further, but no matter the size of $\lambda$, when $x$ crosses the threshold into the non-shaded region, $\Phi(x)$ is left behind by every $l_k$. 

To illustrate how the prime counting function $\pi(x)$ relates to $\li(x)$ when $x$ is contained within an interval $s_k$, we define
\begin{align}
	\li_k(x) &:= \li(x) - \li(p_k^2) \quad \textrm{and} 
	\label{eq:li_and_pi}\\
	\pi_k(x) &:= \pi(x) - \pi(p_k^2), \nonumber
\end{align}
so that both functions are 0 at the start of the interval and  equal to $\li_k$ and $\pi_k$, respectively, at the end of the interval. We contrast the functions in \eqref{eq:li_and_pi} by plotting the error term
\begin{align*}
	\pi_k(x) - \li_k(x)
\end{align*}
across the intervals $s_{500}$, $s_{750}$, and $s_{1000}$, as shown in \figref{fig:counterPlots}. Visually, the wandering of the error term gives the impression of a random walk; it strays away from the theoretical estimate, sometimes smaller, sometimes larger, revealing regions with lower or higher density of primes than suggested by the prime number theorem. Note that in the same plots we have also included the standard deviation $\sigma_k(n)$ of a binomial distribution with probability of success $1/\log {p_{k+1}^2}$, and sequence length $n$, where $1\leq n \leq l_k$, such that
\begin{align*}
	\sigma_k(n) = 
	\sqrt	
	{
		\frac{n}{\log {p_{k+1}^2}}\left( 1 - \frac{1}{\log {p_{k+1}^2}} \right).
	}
\end{align*}
The reason for doing this is discussed in \secref{sec:errorterm}.

Complementary to plotting $\pi_k(x) - \li_k(x)$ we also examine the distribution of gap sizes within an interval. Since the gap distribution is inversely related to the density of primes it helps give an impression of fluctuations in prime density. In \figref{fig:gapsPlots}, we have plotted the prime gaps in each interval $s_{500}$, $s_{750}$, and $s_{1000}$, together with measured average gap size and expected theoretical gap size $g=\log p_{k+1}^2$ (both estimates coincide on the scale of the figure). We also show the moving average of gap sizes with runs taken over 25 elements, demonstrating how the average gap size shifts across the intervals. Note that both \figref{fig:counterPlots} and \figref{fig:gapsPlots} make obvious the presence of significant fluctuations in prime density within the chosen intervals $s_k$. 

Now, we are interested in the specific subintervals within $s_k$ defined by Maier's interval $[x, x+\Phi(x)]$, where $p_k^2\leq x < p_{k+1}^2-\Phi(x)$. We can measure exactly how large the fluctuations in prime density are in these subintervals by calculating the ratio between the number of primes in $[x, x+\Phi(x)]$ and the estimate suggested by the prime number theorem, $\Phi(x)/\log x$. More specifically, we calculate the ratios
\begin{align*}
	\frac{\pi\left(x + (\log x)^3 \right) - \pi(x)}{(\log x)^{2}},
\end{align*}
which we then use to locate an appropriate value of $\delta_3$. The result is displayed in \figref{fig:maierDeltaPlots}, where we have plotted the calculated ratios as $x$ moves across each of the intervals $s_{500}$, $s_{750}$, and $s_{1000}$. The largest value we can choose for $\delta_3$ in order to find examples of $x$ and $X$ in all three intervals that confirm to \eqref{eq:maierex} is $\delta_3\approx 0.064$ (since the minimal largest deviation from 1 across the intervals---found in the interval $s_{500}$---has this value). In the figure, we have chosen $\delta_3=0.03$  as an illustration. In each plot, the shaded region contains the ratio values of Maier's intervals where the density of primes is within a factor $1 \pm \delta_3$ from that predicted by the prime number theorem, while all ratio values falling outside the shaded area correspond to Maier's intervals where the prime density is $1 \pm \delta_3$ times higher or lower  than the theoretical density. 

Also shown in \figref{fig:maierDeltaPlots} is the ratio $\pi_k/(l_k/\log p_{k+1}^2)$, demonstrating how much the prime density of the whole interval $s_k$ varies from the prime number theorem estimate. Note that there is no connection between the increasing lengths of the intervals and the observed increasing ratio $\pi_k/(l_k/\log p_{k+1}^2)$. In fact, if we plot $\pi_k/(l_k/\log p_{k+1}^2)$ across a large range of $k$, as done in \figref{fig:ratioPrimesVsEstimate}, we observe that the convergence of the ratios towards 1 is creepingly slow, with significant fluctuations which magnitude depends on the corresponding gap size $g_k$; the smaller the gap size, the larger the fluctuations, and vice versa.   

An important final note is that we do not show that this $\delta_3$ is the proper choice in Maier's theorem. However, Maier's theorem guarantees that for any $\lambda$, there must be a $\delta_{\lambda}$ so that \eqref{eq:maierex} holds for arbitrarily large $x$ and $X$. In theory therefore, we could for any given $\lambda$ construct an infinite series of ratio plots similar to those in \figref{fig:maierDeltaPlots}, each corresponding to a different interval $s_k$. By choosing the correct $\delta_{\lambda}$ all plots would show the same shaded region, and importantly, all plots would have curves traveling outside this space. 
\end{ex}

\subsection{Shrinking the divide between $\tilde \pi_k$ and $\pi_k$}
\label{subseq:shrink}
Previously, we have encountered two similar estimates for the number of primes within an interval $s_k$. One was the probabilistic estimate derived from Merten's product theorem \eqref{eq:MPNT00},
\begin{align}
	\tilde \pi_k  \sim 2 \e^{-\gamma} \frac{l_k}{\log p_{k+1}^2},
\label{eq:mertens_k_shrink}	
\end{align}
the other was the estimate given by the prime number theorem \eqref{eq:PNT0},
\begin{align*}
\pi_k = \frac{l_k}{\log p_{k+1}^2}.
\end{align*}
The question we ask here is whether it is possible to work from a probabilistic perspective, and then exploit properties of $s_k$ to push $\tilde \pi_k$ closer to $\pi_k$.
While analyzing the sequences $s_k$ in terms of modern sieve theory would be a natural approach here, that is outside the scope of this paper. Rather, we go for a brief stab at the problem---in particular useful for building intuition---to arrive at a few interesting observations.
%

Let us first of all explain what we mean by a probabilistic perspective. In \secref{sec:notation} we learned that the probabilistic prime counting function $\tilde \pi_k$ can be stated as 
\begin{align*}
	\tilde \pi_k =  \frac{1}{p_k\#} \sum_{j=0}^{p_k\#-1}  S(s_k^j,p_k\#),
\end{align*}
while $\pi_k$ is given by
\begin{align*}
	\pi_k = S(s_k^0,p_k\#).
\end{align*}
As our starting point for counting the number of primes in the interval $s_k$, assume the only information available is the length $l_k$ of the interval. Then the only certainty we have is that $\pi_k$ takes one of the values $S(s_k^j,p_k\#)$, $0\leq j < p_k\#$. This is the largest sample space possible and the expected number of primes in the interval is given by $\tilde \pi_k$. Now, clearly, to shrink the gap between $\tilde \pi_k$ and $\pi_k$, we need to reduce this sample space in such a way that it still contains $S(s_k^0,p_k\#)$.

Of course, $S(s_k^0,p_k\#)$ is completely determined by where in the natural numbers the interval $s_k$ is situated (allowing for shifts divisible by $p_k\#$), and in the end, the constraints that will push $\tilde \pi_k$ towards $\pi_k$ must in reflect that.
A very natural constraint is for example that when $n\in s_k$, all elements in $R_k(n)$  must be strictly smaller than $p_{k+1}^2$.
This constraint is easily implemented into the Legendre identity by allowing only denominators smaller than $p_{k+1}^2$:
\begin{align*}
	S(s_k^j,p_k\#) = 
	\lfloor l_k \rceil
	- \sum_{q\in \mathcal P_k} \left\lfloor \frac{l_k}{q} \right\rceil
	+ \sum_{\substack{{q_1, q_2 \in \mathcal P_k} \\ {q_1 < q_2}}}   \left\lfloor \frac{l_k}{q_1 q_2} \right\rceil \nonumber
	+ \dots 
	\pm \sum_{\substack{{q_1, \dots, q_k \in \mathcal P_k} \\ {q_1 < \dots < q_k}\\ {\prod_{q \in \mathcal P_k} q < p_{k+1}^2} }}   \left\lfloor \frac{l_k}{q_1 \dots q_k} \right\rceil.
\end{align*}
Numerically, we demonstrate in \figref{fig:legendreTruncated} that this alone reduces the ratio
\begin{align*}
	\frac{\tilde \pi_k}{l_k/\log p_{k+1}^2} 
\end{align*}
from approximately $2\e^{-\gamma}\approx 1.12$ to approximately $1.03$. 

Now, the Legendre identity is notoriously known for its untamable error term; the unconstrained Legendre identity contains $2^k$ terms, and since each can be rounded either up or down (neglecting the cases of exact divisibility), we arrive at an upper error bound of $2^k$, which rapidly grows out of control as we increase $k$. The result is that it is problematic to make straight forward use of the Legendre identity to count primes, and a look to modern sieve theory is necessary to find methods that work around this complication \cite{Tao:web:2007}. Nonetheless, the constraint we imposed here also influences the error bound in the Legendre identity; while still overwhelming, it shifts from exponential to polynomial growth. We show numerical evidence of this in \figref{fig:legendreError}. 

Another example of constraint is this: In the sequence $s_k$, the first appearance of a composite divisible by $p_i$, for $i<k$, is in one of the positions $p_i - m$, where $m$ is an element in the residue class of \mbox{$p_{j}^2 \imod{p_i}$}, $j\neq i$. (The first appearance of a composite of $p_k$ is of course in position 1, since the first element in $s_k$ is $p_k^2$.) For the primes $p_1,\dots,p_6$, these are the possible positions of first appearance:
\begin{center}	
\begin{tabular}{cl}
Prime              		& Position of first appearance \\
\hline
$p_1$ 				& 2  \\
$p_2$ 				& 3  \\
$p_3$				& 2, 5  \\
$p_4$ 				& 4, 6, 7  \\
$p_5$ 				& 3, 7, 8, 9, 11  \\
$p_6$ 				& 2, 4, 5, 10, 11, 13 
\vspace{1.5mm}
\end{tabular}
\end{center}
Note that this restrains the sequences $\rho_i$, $1\leq i \leq k$, effectively halving the sample space of values $S(s_k^j,p_k\#)$.
Nonetheless, verified by numerical checks, the statistical properties of the reduced sample space do not change significantly from the original one, which we can understand from the fact that the constraint does not incorporate positional information of the sequence $s_k$.

\subsection{The bias of the error term in the prime number theorem}
A well known observation is that the error term in the prime number theorem starts out negative: 
\begin{align*}
\pi(x) - \li(x)<0.
\end{align*}
As we demonstrate in \figref{fig:pix}, this bias continues for the full stretch of our data set, which stops at $x \approx 7.3 \times 10^{13}$. Nonetheless, as Littlewood proved in 1914 \cite{Littlewood:1914}, if we persist in moving $x$ towards infinity, the error term will eventually shift sign, and it will continue to do so infinitely many times over. When that happens for the first time, however, no one knows, though it has been proved that $1.39822\times 10^{316}$ is an upper bound for the event \cite{Bays2000}. A probabilistic measure of the bias was provided in 1994, when Rubinstein and Sarnak \cite{Rubinstein:1994vp}---assuming the Riemann hypothesis--- calculated the logarithmic density of those $x$ such that $\pi(x) - \li(x)>0$ to be around $0.00000026$.

%
%
%
%
%
%
%
%
%
%
%

So, why this bias? While the phenomenon has been widely studied, the literature provides no clear answer as to why there is a bias in the first place. We will not give a definitive answer here, but what we have seen so far strongly suggests that the bias derives from the underlying sieving process. Remember, in the proof of \thmref{thm:MPNT1}, the best approximation of the probabilistic prime counting function $\tilde \pi_k$---when applying Merten's product theorem---is obtained by
\begin{equation*}
	\tilde \pi_k \sim 2 \e^{-\gamma} \frac{l_k}{\log {p_{k+1}^2} }.
\end{equation*}
This choice, rather than for example
\begin{equation*}
	\tilde \pi_k \sim 2 \e^{-\gamma} \frac{l_k}{\log {p_{k}^2} } \quad \textrm{or} \quad \pi_k 
	\sim 2 \e^{-\gamma} \frac{l_k}{\log \frac{1}{2}(p_{k}^2+p_{k+1}^2) }
\end{equation*}
also reflects the fact that sieving by the $k$ first primes completes sieving up to ${p_{k+1}^2}$. 
Similarly, while the prime number theorem guarantees 
\begin{equation*}
	\pi_k 
	\sim \frac{l_k}{\log {p_{k}^2} } 
	\sim \frac{l_k}{\log \frac{1}{2}(p_{k}^2+p_{k+1}^2) }
	\sim \frac{l_k}{\log {p_{k+1}^2} },
\end{equation*}
the latter statement---that each sieve step $k$ sieves up to ${p_{k+1}^2}$---suggests that of these three estimates for the prime counting function $\pi_k$, the most accurate is
\begin{equation*}
	\pi_k 
	\sim \frac{l_k}{\log {p_{k+1}^2} }.
\end{equation*}

Now, since $\log p_{k}^2 < \log t < \log p_{k+1}^2$ whenever $p_k^2< t < p_{k+1}^2$, we have that 
\begin{equation*}
\frac{l_k}{\log {p_{k+1}^2} } < \int_{p_k^2}^{p_{k+1}^2} \frac{dt}{\log t} < \frac{l_k}{\log {p_{k}^2} }. 
\end{equation*}
And then, summing over all intervals $s_j$, $j\geq1$, we get 
\begin{equation*}
\sum_{j=1}^{k} \frac{l_j}{\log {p_{j+1}^2} }<\li(p_{k+1}^2) < \sum_{j=1}^{k} \frac{l_j}{\log {p_{j}^2} }.
\end{equation*}
Therefore, under the assumption that 
\begin{equation*}
\pi(p_{k+1}^2) \sim \sum_{j=1}^{k} \frac{l_j}{\log {p_{j+1}^2} }
\end{equation*}
is our most accurate estimate of the prime counting function, the reason why $\pi(x) - \li(x)$ is leaning towards negative becomes apparent:
using $\li(p_{k+1}^2)$ to estimate $\pi(p_{k+1}^2)$, 
rather than $\sum_{j=1}^{k} l_j / \log p_{j+1}^2$, gives a negative expectation value for the error term:
%
%
\begin{equation*}
E[\pi(x)-\li(x)]<0.
\end{equation*}
Following this logic, it would be natural to suggest that 
\begin{equation*}
	E[\pi(p_{k+1}^2)-\sum_{j=1}^{k} \frac{l_j}{\log {p_{j+1}^2} }]=0.
\end{equation*}
However, as we see in \figref{fig:pix}, the statistical evidence seems to indicate that the truth lies somewhere in between,  
\begin{align*}
	\sum_{j=1}^{k} \frac{l_j}{\log {p_{j+1}^2} }<E[\pi(p_{k+1}^2)]<\li(p_{k+1}^2).
\end{align*}
To gain a more accurate impression of this observation, we normalize the curves in \figref{fig:pix} by 
\begin{equation*}
	\Delta_k :=  \frac{1}{2}\sum_{j=1}^{k} \left( \frac{l_j}{\log {p_{j}^2} } - \frac{l_j}{\log {p_{j+1}^2} }\right),
\end{equation*}
which places the differences 
\begin{equation*}
	\sum_{j=1}^{k} \left( \frac{l_j}{\log {p_{j+1}^2} } - \li_j \right) \quad \textrm{and} \quad \sum_{j=1}^{k} \left( \frac{l_j}{\log {p_{j}^2} } - \li_j \right)
\end{equation*}
approximately at the constant lines -1 and 1 respectively. The resulting plot is shown in \figref{fig:pixnormalized}. There we observe that on the scale of the data available, $[\pi(p_{k+1}^2)-\li(p_{k+1}^2)]/\Delta_k$ apparently fluctuates fairly stable around a mean value of $-0.60$. Constructing an empirical probability density function, we see in \figref{fig:pixgauss} that across the full interval $x$, the fluctuations closely resembles a Gaussian distribution with mean and standard deviation given by $-0.60$ and $0.12$. 

We will not attempt an explanation to why the bias of the error term lies where it does in these observations, or whether it will continue to do so. However, these directions of thoughts seems to add valuable clues towards a rigorous understanding of the error term.

\subsection{An upper bound for the error term in the prime number theorem}
\label{sec:errorterm}
Despite the fact that the number of primes up to $x$ is deterministic, let us for a moment assume that $\pi(x)$ can be represented by a random variable $\Pi(x)$ with 
mean value $\li(x)$ and variance $\sigma^2(x)$. Then we could interpret the error term $|\pi(x) - \li(x)|$ in terms of the standard deviation
%
\begin{equation*}
\sigma(x) = \sqrt{E\left[(\Pi(x) - \li(x))^2\right]}.
\end{equation*}
In this section we will walk through a strategy for obtaining an upper bound on the error term based on this view. We start with formulating a random model for the number of primes within $s_k$.
%

From \secref{sec:notation}, we know that the number of coprimes to $p_k\#$ within any interval of length $l_k$ takes one of the values $S(s_k^j,p_k\#)$, $0\leq j < p_k\#$, and specifically, we have for $s_k$ that  $\pi_k=S(s_k^0,p_k\#)$. This naturally suggests a random model for $\tilde \pi_k$: 
%
%
%
%
%
%
Let the number of coprimes to $p_k\#$ within any interval of length $l_k$ be described by the random variable $\tilde \Pi_k$, with sample space $S(s_k^j,p_k\#)$, $0\leq j < p_k\#$. Using the approximation 
\begin{align*}
	l_k \cdot \prod_{p \in \mathcal P_k} \left(1-\frac{1}{p}\right) \sim 2\e^{-\gamma}\frac{l_k}{\log p_{k+1}^2}
\end{align*}
(and neglecting the error term) we can write the mean value of $\tilde \Pi_k$ as 
\begin{align*}
	\tilde \mu_k := E[\tilde \Pi_k] = \frac{1}{p_k\#} \sum_{j=0}^{p_k\#-1 }S^j(s_k,p_k\#) =
	2\e^{-\gamma}\frac{l_k}{\log p_{k+1}^2},
\end{align*}
and likewise the variance
\begin{align*}
	\tilde \sigma^2_k  := E[(\tilde \Pi_k - \tilde \mu_k)^2]= \frac{1}{p_k\#} \sum_{j=0}^{p_k\# -1} \left(S^j(s_k,p_k\#) - 2\e^{-\gamma}\frac{l_k}{\log p_{k+1}^2}\right)^2.
\end{align*}
With a slight adjustment to this model---multiplying all elements in the sample space by $\e^{\gamma}/2$---the mean value changes to that of the prime number theorem and we arrive at a  random model for $\pi_k$. 
%
%
%
%
%
Therefore, let the number of primes in $s_k$ be represented by the random variable $\Pi_k$, with mean value
\begin{align*}
	\mu_k = \frac{\e^{\gamma}}{2} \tilde \mu_k = \frac{l_k}{\log p_{k+1}^2},	
\end{align*}
and variance
\begin{align*}
	\sigma^2_k = \frac{\e^{2\gamma}}{4} \tilde \sigma^2_k 
	= \frac{1}{p_k\#} \sum_{j=0}^{p_k\# -1} \left(\frac{\e^{\gamma}}{2}  S^j(s_k,p_k\#) - \frac{l_k}{\log p_{k+1}^2}\right)^2.
\end{align*}

%
%
%
%
%
One thing to note is that the mean value is an explicit function of $l_k$ and $p_{k+1}$, while this is not the case with the variance. Rather than aiming for deriving an exact expression, we go about this by finding another distribution where both mean and variance are known functions, and which variance bounds $\sigma^2_k$. An obvious candidate is the binomial distribution with number of trials given by $l_k$, and probability of success  $1/\log p_{k+1}^2$. Therefore, let $\Pi_{k, \textrm{B}}$ be the binomial random variable,
\begin{align*}
	\Pi_{k, \textrm{B}} \sim 	\textrm{B}\left(l_k, \frac{1}{\log p_{k+1}^2}\right).
\end{align*}
Intuitively, there are two reasons why $\Pi_{k, \textrm{B}}$ should have larger variance than $\Pi_{k}$; one is that the sample space of our $\pi_k$ random model is only a very small fraction of the sample space of the binomial model---$p_k\#$ elements versus $2^{l_k}$; a second is that the $\pi_k$ random model has strict upper and lower bounds, while the binomial model allows $\Pi_{k, \textrm{B}}$ to take any values from $0$ to $l_k$. These reasons are not proof of course, but for now we assume it is true that the variance of $\Pi_{k, \textrm{B}}$ bounds  $\sigma^2_k$. (This seems a fairly simple problem, and we expect theory to exist that applies. It is possible to construct a version of Polya's urn model that works heuristically, but in the end the approach mirrors an inclusion-exclusion process and does not provide proof).

%
%
%
%

Since we are mostly interested in the limit of $\toinf{k}$, where the binomial distribution approaches a Poisson distribution, we can replace the binomial variable $\Pi_{k, \textrm{B}}$ by the Poisson random variable $\Pi_{k, \textrm{Pois}}$, with parameter $l_k/\log p_{k+1}^2$,
\begin{align*}
	\Pi_{k, \textrm{Pois}} \sim 	\textrm{Pois}\left(\frac{l_k}{\log p_{k+1}^2}\right).
\end{align*}
%
The mean and variance of $\Pi_{k, \textrm{Pois}}$ are identical and given by
\begin{align*}
	\mu_{k,{\textrm{Pois}}} = \sigma_{k,{\textrm{Pois}}}^2 = \frac{l_k}{\log p_{k+1}^2},
\end{align*}
and it follows by assumption that the variance of the $\pi_k$ random model satisfies
\begin{align*}
	 \sigma_{k}^2  < \sigma_{k,{\textrm{Pois}}}^2 = \frac{l_k}{\log p_{k+1}^2}.
\end{align*}
%
%
%
%
%

%

%
Now, let us turn to $\pi(x)$. Since $\pi(x)$ can be written as the sum  
\begin{align*}
	\pi(x) = \sum_{j=1}^{k-1} \pi_j + \frac{x-p_k^2}{l_k} \, \pi_k, 
\end{align*}
where $k$ is the integer such that $p_k^2 \leq x < p_{k+1}^2$, 
it makes sense to define a random model for $\pi(x)$ in the same way. Hence, let $\Pi(x)$ be the random variable given by
\begin{align*}
	\Pi(x) := \sum_{j=1}^{k-1} \Pi_j + \frac{x-p_k^2}{l_k} \, \Pi_k. 
\end{align*}
%
%
%
%
%
%
%
%
%
Since $\Pi(x)$ is defined as a sum over random independent variables, its mean and variance are obtained simply by summing over the individual means and variances:
\begin{align*}
	\mu(x) := E[\Pi(x)] 	&= \sum_{j=1}^{k-1} E[\Pi_j] + \frac{x-p_k^2}{l_k} \, E[\Pi_k] \\
					&= \sum_{j=1}^{k-1} \frac{l_j}{\log p_{j+1}^2} + \frac{x-p_k^2}{l_k} \, \frac{l_k}{\log p_{k+1}^2}\\
					&\approx \li(x),
\end{align*}
and
\begin{align*}
	\sigma^2(x)  := 	E[\left(\Pi(x) - \mu(x)\right)^2] 
	&= \sum_{j=1}^{k-1} E[(\Pi_j - \mu_j)^2]  + \frac{x-p_k^2}{l_k} \, E[(\Pi_k - \mu_k)^2] \\\\
	&< \sum_{j=1}^{k-1} \sigma_{j,\textrm{Pois}}^2 + \frac{x-p_k^2}{l_k} \, \sigma_{k,\textrm{Pois}}^2 \\
	&= \sum_{j=1}^{k-1} \frac{l_j}{\log p_{j+1}^2} + \frac{x-p_k^2}{l_k} \, \frac{l_k}{\log p_{k+1}^2}\\
	&\approx \li(x).
\end{align*}
Consequently, the standard deviation in the $\pi(x)$ random model satisfies
\begin{align*} 	
	\sigma(x)  = 	\sqrt{E[\left(\Pi(x) - \mu(x)\right)^2]} < \sqrt{\li(x)}.
\end{align*}

Of course, the most important question remains: Does the properties of the $\pi(x)$ random model translate to $\pi(x)$ proper? For the answer to be positive, there are essentially two requirements that must be fulfilled. Firstly, we must be able to interpret the number of primes $\pi_k$ in each interval $s_k$ as a random variable with well defined mean and variance. 
Secondly, we must be able to view $\pi(x)$ as a random variable, with mean and variance derived by summing over the means and variances of the individual intervals $s_k$. 

We have already seen that the first requirement is satisfied; $\pi_k$ is always in the sample space of $\tilde \Pi_k$ and can be viewed as drawn from this set (or alternatively using the  
$\pi_k$ model as we do above; the models differ only by a multiplicative constant). Regarding the second requirement; the distribution of prime numbers is deterministic, and therefore the assumption of independent random variables $\pi_k$ does not hold. While this does not affect constructing the mean of $\pi(x)$ as a sum of the individual means, we need to be more careful about the variance. Fortunately, it is enough to show that if the $\pi_k$s can be considered to be uncorrelated, i.e.,
\begin{align*}
	\rho_{ij} := \frac{\textrm{E}[(\pi_i - l_i/\log p_{i+1})(\pi_j - l_j/\log p_{j+1})]}
			{\textrm{E}[(\pi_i - l_i/\log p_{i+1})^2]}
			=0,
			\quad i \neq j,
\end{align*}
that---according to the Bienaym\'e formula---will be enough for us to sum variances linearly. While we do not rigorously prove uncorrelation, we proceed with a few heuristic arguments for why it very likely holds.

It will be helpful to start by considering how the values $\pi_k$ and $\pi_{k+1}$ of two neighboring intervals $s_k$ and $s_{k+1}$ relate when $k$ is large. The first thing to note is that we can safely assume that the lengths of the intervals, $l_k$ and $l_{k+1}$, are much smaller than the period $p_k\#$ of $R_k(n)$, since the latter grows much faster with $k$. Now, suppose $l_k=l_{k+1}$  and that both $s_k$ and $s_{k+1}$ are sieved by the $k$ first primes (of course, none of which are true). Then we would have 
\begin{align*}
	\pi_k = S(s_k^0,p_k\#) \quad \textrm{and} \quad 	\pi_{k+1} = S(s_k^{l_k},p_k\#).
\end{align*}
Under these assumptions we see that $\pi_k$ and $\pi_{k+1}$ each correspond to different elements of the full sample space of $\tilde \Pi_k$. What's more, the two elements are derived by counting the coprimes to $p_k\#$ within two non-overlapping subintervals of the much longer periodic sequence $R_k(n)$. We would therefore expect the correlation between $\pi_k$ and $\pi_{k+1}$ to be weak already under these considerations; possibly could this be proved in terms of the autocorrelation of $R_k(n)$, which we would expect to drop off to zero at a distance of $l_k$. Naturally, $l_k$ and $l_{k+1}$ are never equal, and often differ significantly because of the distribution of gaps $g_k$. In addition, $s_{k+1}$ is sieved by the additional prime $p_{k+1}$. Both of these factors serve to decrease correlation, and in the case when we look at intervals further apart, the effects are magnified. Heuristically therefore, we expect the $\pi_k$s to behave as uncorrelated variables and we can apply the Bienaym\'e formula. 

Taking what we have seen to its conclusion, we conjecture an upper bound for the error term in the prime number theorem:

\begin{conjecture} As $x\rightarrow \infty$, the error term in the prime number theorem satisfies
\begin{align*}
|\pi(x) - \li(x)| < \sqrt{\li(x)}.
\end{align*}
\end{conjecture}
\noindent From the above---and from what we understand---it is also clear that the two remaining obstacles for a rigorous proof of the conjecture is 1) to prove that the $\pi_k$s
can be considered uncorrelated; and 2) to prove that $\sigma^2_k<\sigma^2_{k,\textrm{Pois}}$. We expect neither of these proofs to have strong barriers, nor that any surprises should arise.
\subsubsection{Empirical evidence}
As we have just seen, underlying the suggested conjecture is the interpretation that the $\pi_k$s behave as random, uncorrelated variables, with well defined means and variances. And additionally, that we can employ the variance of a binomial distribution to bound the variance of $\pi_k$. To round off the paper, we here present empirical results supporting this picture, starting plainly with the set of $\pi_k$s.


 From the prime number theorem we know that 
\begin{align*}
	\pi_k  \sim \frac{l_k}{\log p_{k+1}^2} = \frac{2p_{k+1}g_k - g_k^2}{\log p_{k+1}^2}.
\end{align*}
Therefore, by substituting $x=p_{k+1}^2$, we would expect all values of $\pi_k$ to lie close to one of the curves 
\begin{align*}
	\pi(x,g) := \frac{2 \sqrt{x}g-g^2}{\log x}, \quad g=2,4,\dots,
\end{align*}
which is demonstrated to hold in \figref{fig:pik}. Now, to gain an impression of how the $\pi_k$s fluctuate around each curve, 
we plot all values $\pi_k - \li_k$. At first eyesight then, \figref{fig:diffpiklik} reveals that these values are seemingly evenly distributed around 0, confirming to the behavior of a random variable. By sampling across values of $\pi_k - \li_k$ we arrive at a set of empirical probability density functions and find that the fluctuations are well described by Gaussian distributions, both for samples taken across $\pi_k$s corresponding to different curves $\pi(x,g)$ and for samples taken along  the curve $\pi(x,6)$ across increasing $k$; all shown in \figref{fig:distributionsdiffpiklik}.

We continue with the correlation between the number of primes $\pi_{k}$ and $\pi_{k+j}$ in intervals that differ by $j$ sieve steps. In \figref{fig:correlationsdiffpiklik}A we plot the empirical correlation as calculated from 
\begin{align*}
	\rho_{ij} := \frac{\textrm{E}[(\pi_i - l_i/\log p_{i+1})(\pi_j - l_j/\log p_{j+1})]}
			{\textrm{E}[(\pi_i - l_i/\log p_{i+1})^2]}
			=0,
			\quad i \neq j,
\end{align*}
with the average taken over all values $1 \leq k \leq p_{k+1}^2$. What we observe is that for $j>50$, the correlation fluctuates evenly around 0. However, for smaller values of $j$, and in particular for $j=1$---corresponding to neighboring intervals---we note a slight anti-correlation. This clearly expresses itself in \figref{fig:correlationsdiffpiklik}B, where the correlation is calculated for intervals of $10000$ non-overlapping samples of $k$. Here we see that for increasing interval number and hence larger $k$, there is a visible trend towards stronger anti-correlation of close neighbors. We do not currently have an explanation for this observation. 
Nonetheless, if the effect of this behavior is significant, it will in fact reduce the variance of $\pi(x)$, since anti-correlated variables contribute negatively to the total variance, as seen from the identity
\begin{align*}
	\sigma^2(p_{k+1}^2)
	= \sum_{i=1}^k \sum_{j=1}^k \sigma_i \sigma_j
	= \sum_{i=1}^k \sigma_i^2 + 2\sum_{1\le i}\sum_{<j\le k} \sigma_i \sigma_j. 
\end{align*}

Consider now the assumption that the variance of the binomial distribution $\textrm{B}(l_k,l_k/\log p{k+1}^2)$ can bound the variance of $\pi_k$. We describe this assumption by three different examples, on somewhat different length scales. To start with the smallest scale, we view in \figref{fig:counterPlots} how the standard deviation of the binomial distribution compares to $\pi_k(x)-\li_k(x)$ across the three intervals $s_{500}$, $s_{750}$, and $s_{1000}$. While the fluctuations of $\pi_k(x)-\li_k(x)$ are large, the steady growth of the binomial standard deviation ensures it ends up with the highest value at $x=p_{k+1}^2$. 

Moving up in scale, we now take the point of view of probability density functions (PDF). For this purpose, 
%
%
%
%
we have simulated the $\pi_k$ random model for the interval $s_{50}$, and plotted the measured PDF of $\Pi_{50}-\li_{50}$ together with the theoretical binomial PDF for that interval (\figref{fig:distbin}A). In parallel, we have sampled values of $\pi_k-\li_k$ for intervals $s_k$ with $g_k=6$. The resulting PDF is shown in \figref{fig:distbin}B together with the two binomial  PDFs corresponding to the start and end of the sample interval. In both the simulated and real case, we see that the binomial distribution has the widest PDF, and hence greatest variance.

Finally, we compare standard deviations of binomial distributions with standard deviations calculated from sampled data,
across all intervals $s_k$, $1\leq k \leq 575200$.
This is shown in \figref{fig:sampledstdev} for those $s_k$ such that $g_k=6$. Not only are the standard deviations of the binomial distributions always larger than those of the $\pi_k$s, but the gap increases as $\toinf{k}$.

What we find then is convincing numerical evidence of our claim that the $\pi_k$s can be interpreted as a set of random, uncorrelated variables. This implies, of course, that we should expect to observe empirically that $|\pi(x) - \li(x)|<\sqrt{\li(x)}$ as $\toinf{x}$. This is apparent in \figref{fig:pix}, where $\pi(x) - \li(x)$ wiggles towards infinity, well within the range of the conjectured error bound.

\bibliography{primereferences}

\begin{thebibliography}{10}

\bibitem{Bays2000}
C~Bays and R~H Hudson.
\newblock {A new bound for the smallest $x$ with $\pi(x)>{li}(x)$}.
\newblock {\em Mathematics of Computation}, 69:1285--1296, 2000.

\bibitem{Poussin:1899}
C~J de~la Vall\'ee-Poussin.
\newblock {Recherches analytiques sur la thŽorie des nombers premiers}.
\newblock {\em Annales de la SociŽtŽ Scientifique de Bruxelles}, 20:183--256,
  1899.

\bibitem{DIAMOND:ug}
H~G Diamond.
\newblock {Elementary Methods in the Study of the Distribution of
  Prime-Numbers}.
\newblock {\em Bulletin of the American Mathematical Society}, 7:553--589,
  1982.

\bibitem{Fine:2010vu}
B~Fine and G~Rosenberger.
\newblock {An epic drama: the development of the prime number theorem}.
\newblock {\em Scientia. Series A. Mathematical Sciences. New Series},
  20:1--26, 2010.

\bibitem{Goldfeld:2004wh}
D~Goldfeld.
\newblock {The Elementary Proof of the Prime Number Theorem: an Historical
  Perspective}.
\newblock In {\em Number Theory}, pages 179--192. Springer, 2004.

\bibitem{Goldston:2007wk}
D~A Goldston.
\newblock {Are There Infinitely Many Twin Primes?}
\newblock {\em arXiv preprint arXiv:0710.2123}, 2007.

\bibitem{Granville:1995cm}
A~Granville.
\newblock {Harald Cram{\'e}r and the distribution of prime numbers}.
\newblock {\em Scandinavian Actuarial Journal}, 1995(1):12--28, January 1995.

\bibitem{Granville:2010cq}
A~Granville.
\newblock {Different approaches to the distribution of primes}.
\newblock {\em Milan journal of mathematics}, 78(1):65--84, 2010.

\bibitem{Greaves:2001}
G~Greaves.
\newblock {\em {Sieves in Number Theory}}.
\newblock Springer, 2001.

\bibitem{Havil:2003}
J~Havil.
\newblock {\em {Gamma: exploring Euler's constant}}.
\newblock Princeton University Press, 2003.

\bibitem{Littlewood:1914}
J~W Littlewood.
\newblock {Distribution des Nombres Premiers}.
\newblock {\em C. R. Acad. Sci. Paris}, 158:1869--1872, 1914.

\bibitem{Maier:1985ww}
H~Maier.
\newblock {Primes in short intervals.}
\newblock {\em The Michigan Mathematical Journal}, 32(2):221--225, 1985.

\bibitem{Rubinstein:1994vp}
M~Rubinstein and P~Sarnak.
\newblock {Chebyshev's bias}.
\newblock {\em Experimental Mathematics}, 3(3):173--197, 1994.

\bibitem{Selberg:1943}
A~Selberg.
\newblock {On the normal density of primes in small intervals and the
  difference between consecutive primes}.
\newblock {\em Archiv for mathematik og naturvidenskab}, 47:87--105, 1943.

\bibitem{Tao:web:2007}
T~Tao.
\newblock Open question: The parity problem in sieve theory.
\newblock
  \url{http://terrytao.wordpress.com/2007/06/05/open-question-the-parity-problem-in-sieve-theory},
  2007.

\bibitem{Villarino:2005to}
M~B Villarino.
\newblock {Mertens' Proof of Mertens' Theorem}.
\newblock {\em arXiv preprint arXiv:math/0504289}, 2005.

\bibitem{Yildirim:1999ws}
C~Y Yildirim.
\newblock {A survey of results on primes in short intervals}.
\newblock {\em Lecture notes in pure and applied mathematics}, pages 307--343,
  1999.

\bibitem{Yildirim:2009vv}
C~Y Yildirim.
\newblock {The distribution of primes: conjectures vs. hitherto provables}.
\newblock In {\em Further progress in analysis}, pages 75--108. World Sci.
  Publ., Hackensack, NJ, 2009.

\end{thebibliography}
\bibliographystyle{plain}
\clearpage

\appendix

\section*{Appendix: Figures}
\begin{figure}[h]
\centering
\includegraphics[width=\textwidth]{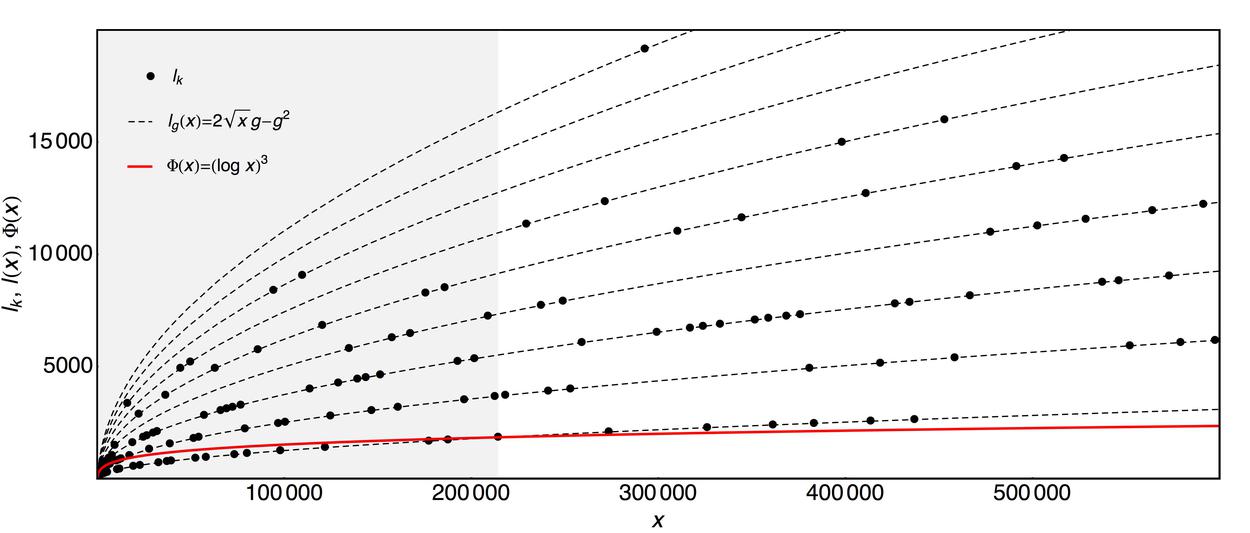}
\caption{
\small \sl
The lengths $l_k=2p_{k+1}g_k-g_k^2$ of the intervals $s_k$ plotted at the values $x = p_{k+1}^2$, $1\leq k \leq 136$ (black circles). All points lie on the curves $l(x)=2 \sqrt{x}g-g^2$ (dashed lines), where $g=2,4,\dots$. The points on the lowest line all correspond to values of $k$ such that $g_k = 2$; the points on the second lowest to those $k$ such that $g_k = 4$, and so on. In addition, $\Phi(x) = (\log x)^{3}$---the length of Maier's interval with $\lambda=3$---is plotted as a continuous curve (red line). In the shaded region $\Phi(x)$ is large enough to cover some intervals $s_k$ completely, while beyond this region, $\Phi(x)$ is strictly smaller than any $l_k$. 
}
\label{fig:interval_lengths}
\end{figure}

\begin{figure}
        \centering
        \begin{subfigure}[b]{\textwidth}
                \centering
		\includegraphics[width=\textwidth]{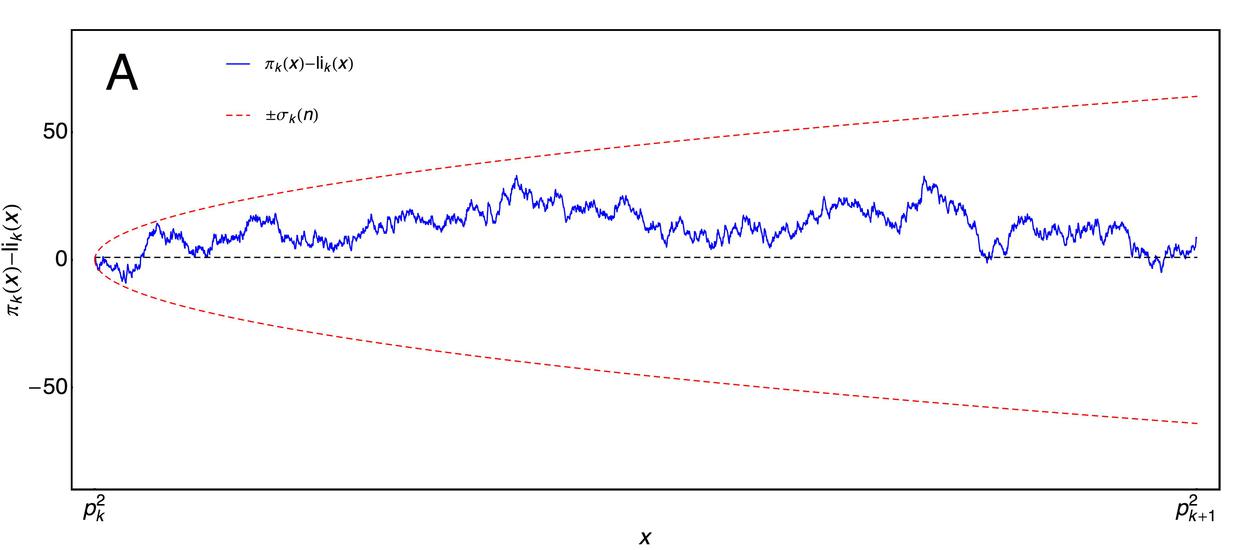}
        \end{subfigure}\\%
        \begin{subfigure}[b]{\textwidth}
                \centering
		\includegraphics[width=\textwidth]{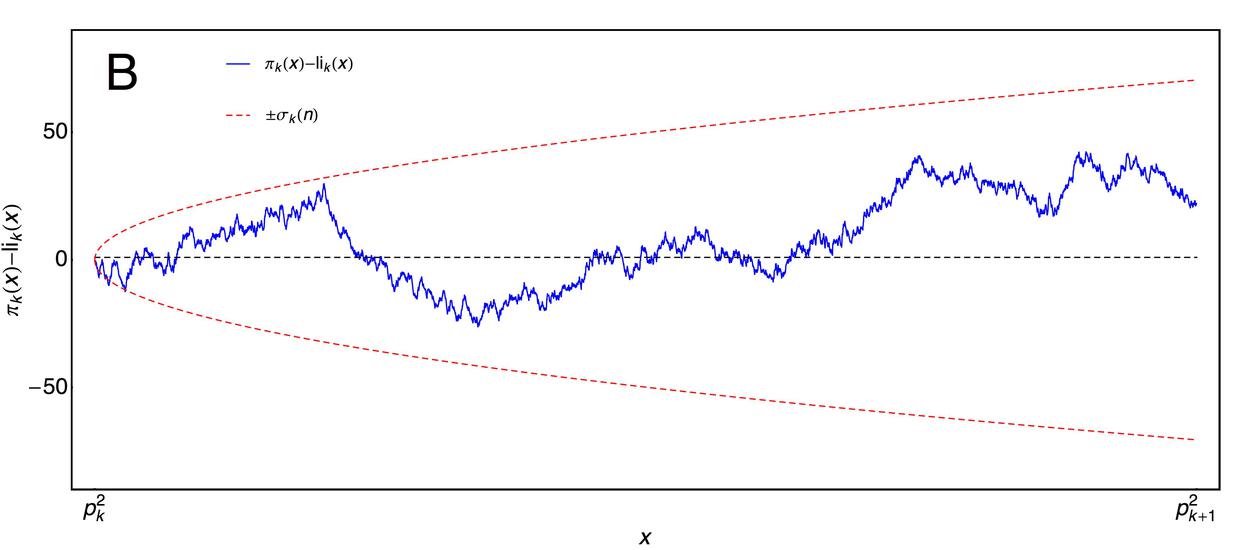}
        \end{subfigure}\\
        \begin{subfigure}[b]{\textwidth}
                \centering
		\includegraphics[width=\textwidth]{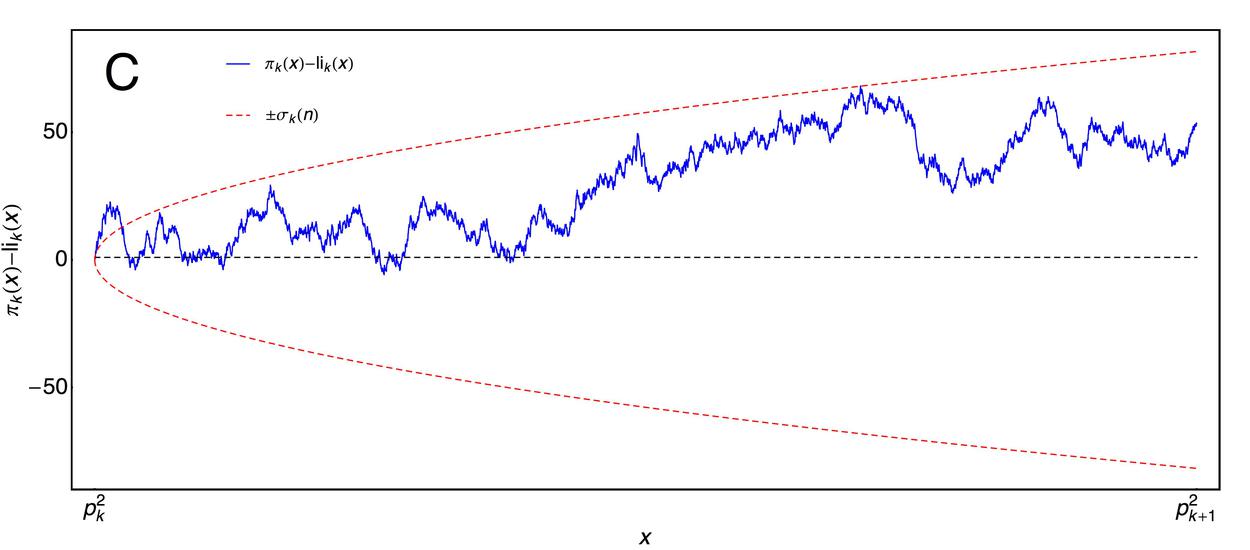}
        \end{subfigure}
\caption{
\small \sl
$\pi_k(x) - li_k(x)$ (blue) plotted across the interval $s_k$ for three values of $k$: A) $k=500$, B)  $k=750$, and C)  $k=1000$. Also shown is the standard deviation $\sigma_k(n)$ (red, dashed) of a binomial distribution with probability of success $p=1/\log {p_{k+1}^2}$, and sequence length $n$, where $1\leq n \leq l_k$. The different intervals have lengths $l_{500}=71250$, $l_{750}=91152$, and $l_{1000}=126768$. The intervals correspond to the gaps $g_{500}=10$, $g_{750}=8$, and $g_{1000}=8$, where $g_k:=p_{k+1}-p_{k}$.
}
\label{fig:counterPlots}
\end{figure}

\begin{figure}
        \centering
        \begin{subfigure}[b]{\textwidth}
                \centering
		\includegraphics[width=\textwidth]{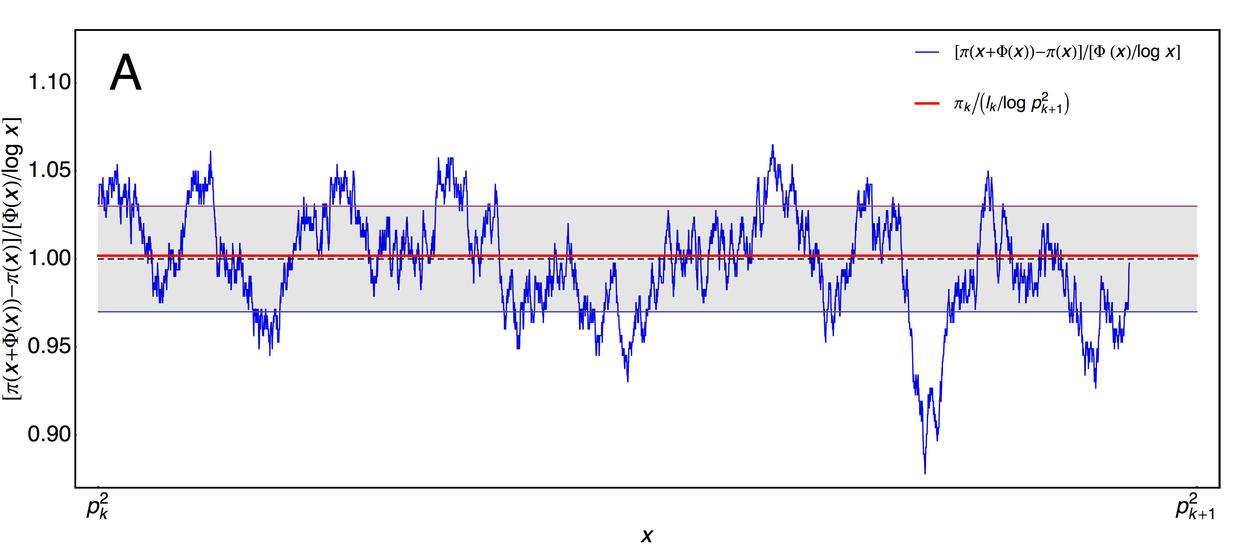}
        \end{subfigure}\\%
        \begin{subfigure}[b]{\textwidth}
                \centering
		\includegraphics[width=\textwidth]{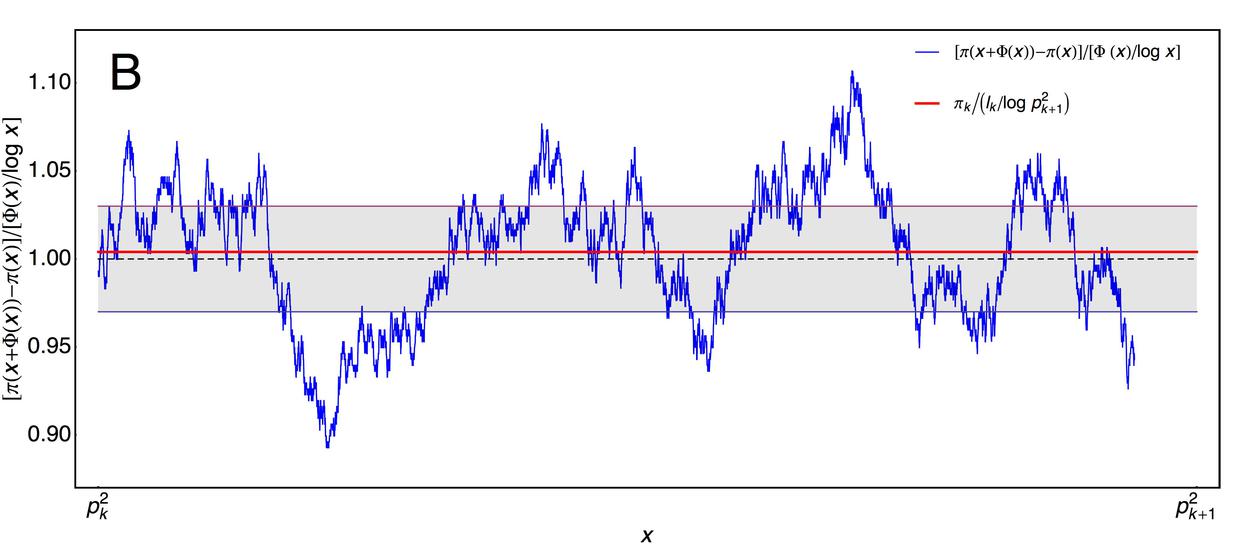}
        \end{subfigure}\\
        \begin{subfigure}[b]{\textwidth}
                \centering
		\includegraphics[width=\textwidth]{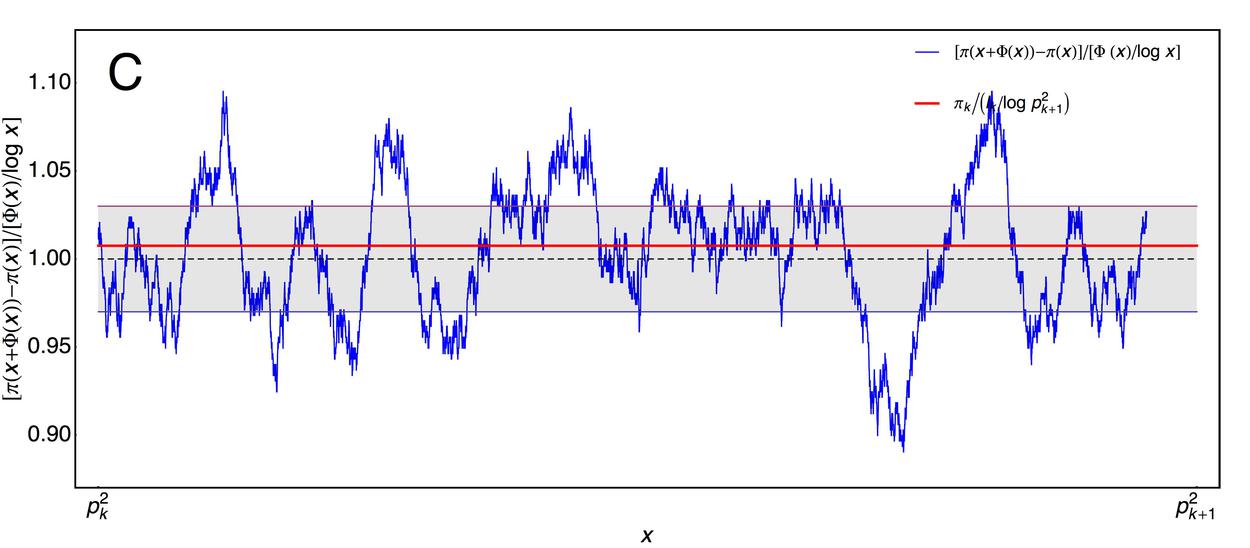}
        \end{subfigure}
\caption{
\small \sl
The prime gaps $g_i := p_{i+1}-p_{i}$ within $s_k$, plotted at $p_i$ for all $i$ such that $p_k^2\leq p_i<p_{i+1}<p_{k+1}^2$ (gray). Also shown are the measured average gap length (red) and the expected gap length in the interval given by the prime number theorem (blue, dashed), in addition to a moving average taken over runs of 25 gap lengths (black). Each plot corresponds to a specific value value of $k$: A) $k=500$, B)  $k=750$, and C)  $k=1000$. 
}
\label{fig:gapsPlots}
\end{figure}

\begin{figure}
        \centering
        \begin{subfigure}[b]{\textwidth}
                \centering
		\includegraphics[width=\textwidth]{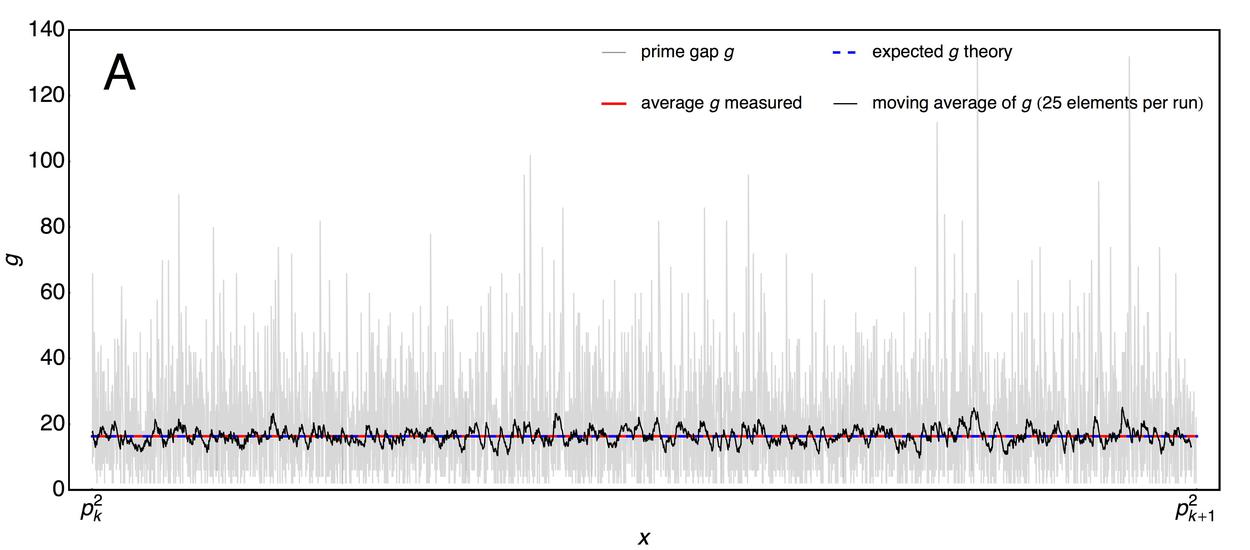}
        \end{subfigure}\\%
        \begin{subfigure}[b]{\textwidth}
                \centering
		\includegraphics[width=\textwidth]{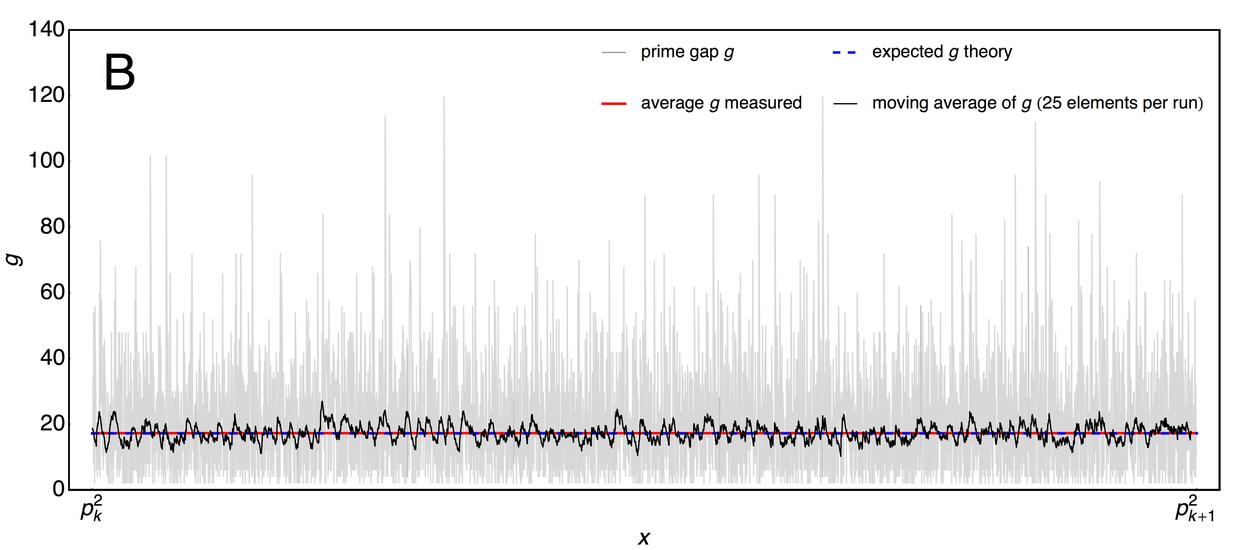}
        \end{subfigure}\\
        \begin{subfigure}[b]{\textwidth}
                \centering
		\includegraphics[width=\textwidth]{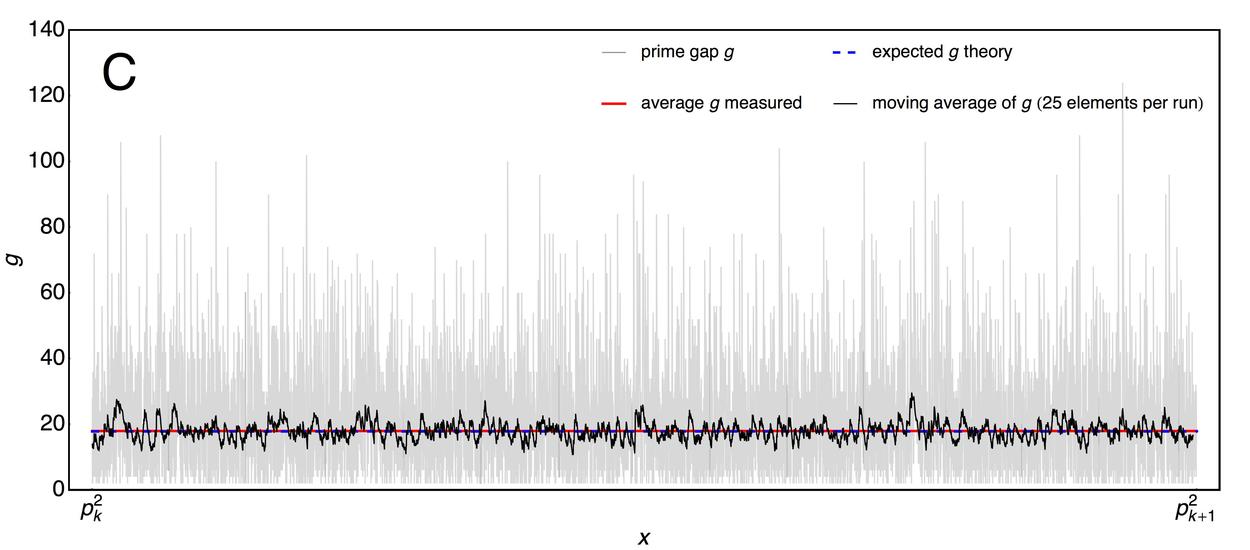}
        \end{subfigure}
\caption{
\small \sl
The ratio between the actual number of primes within a subinterval $[x, x+\Phi(x)]$ of $s_k$ and the estimate suggested by the prime number theorem, plotted for different values of $x$ (blue, rugged lines). The lengths of the subintervals are given by $\Phi(x) = (\log x)^{3}$. The ratio values are plotted at $x$, with $p_k^2 \leq x < p_{k+1}^2-\Phi(x)$, which explains the last empty stretch in each plot. All ratio values climbing above the shaded area correspond to intervals where the density is 1.003 times higher than that predicted by the prime number theorem. Likewise, the values falling below the shaded area reveal the intervals where the density is 0.97 times lower than what the prime number theorem predicts. 
The red line shows the ratio obtained when using the length $l_k$ of the whole interval $s_k$ rather than $\Phi(x)$. The different plots are for the intervals $s_k$, where A) $k=500$, B)  $k=750$, and C)  $k=1000$. 
}
\label{fig:maierDeltaPlots}
\end{figure}

\begin{figure}[h]
\centering
	\includegraphics[width=\textwidth]{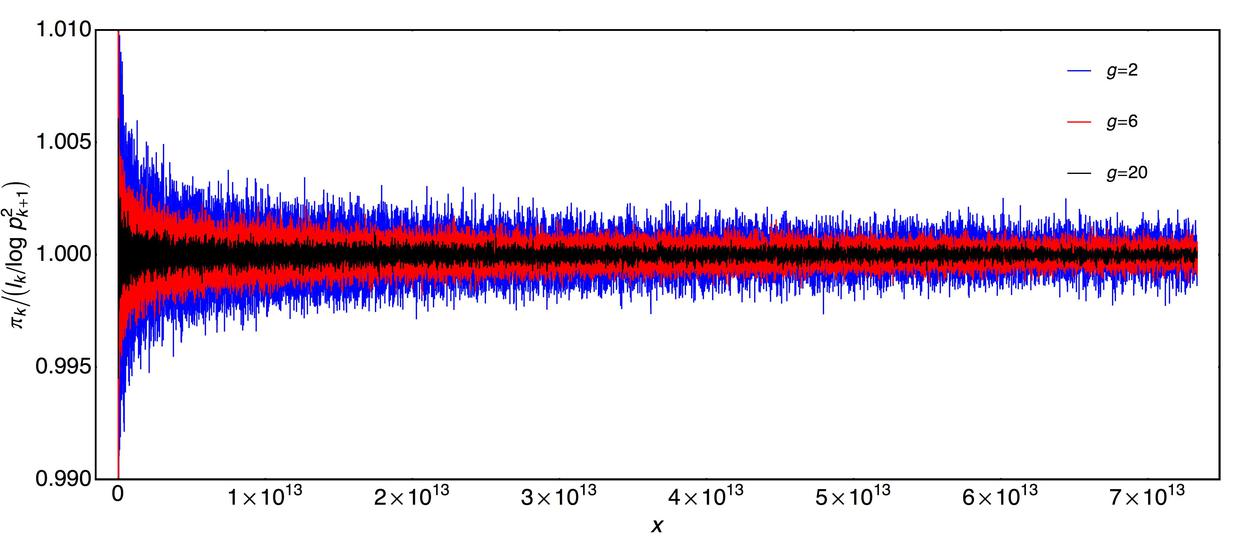}
\caption{
\small \sl
The ratio between the number of primes $\pi_k$ within the interval $s_k$, and  the estimate $l_k/\log p_{k+1}^2$ suggested by the prime number theorem. The different curves 
show the ratio values of the intervals $s_k$ corresponding to the prime gaps $g_k:=p_{k+1}-p_{k} = g$, where $g$ take the values 2 (blue), 6 (red), and 20 (black). $k$ runs from 1 to 575200, but only the values of $k$ corresponding to gaps 2, 6, and 20 are shown.
}
\label{fig:ratioPrimesVsEstimate}
\end{figure}

\begin{figure}
	\centering
	\includegraphics[width=\textwidth]{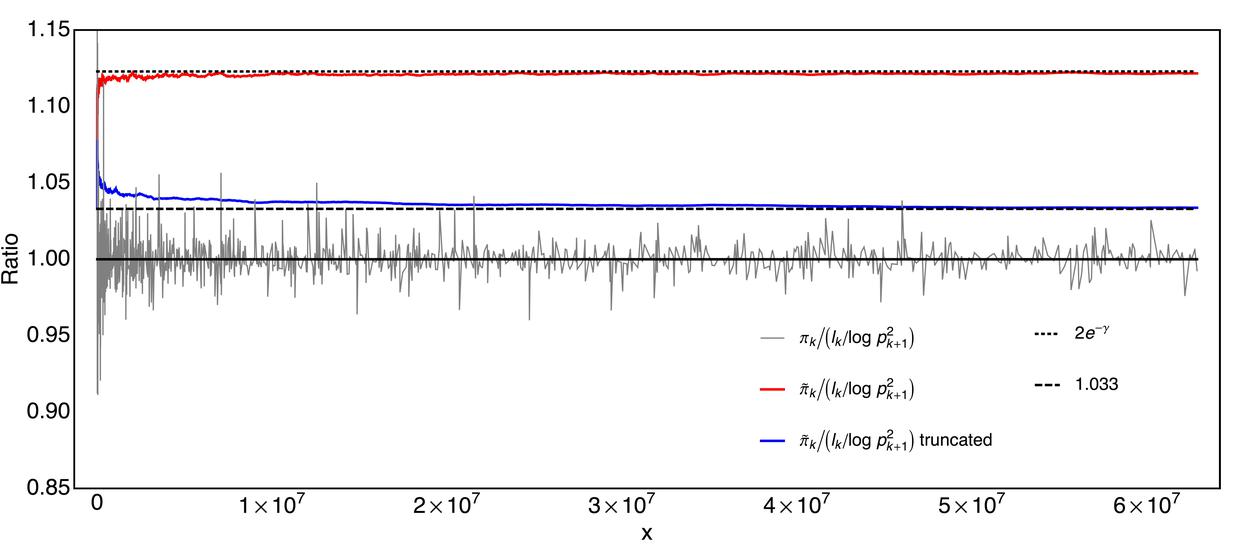}
\caption{
\small \sl
The ratios between the values of three different prime counting functions over the interval $s_k$ and the number of primes in $s_k$ estimated by the prime number theorem, $l_k/\log p_{k+1}^2$. The values are plotted at $x=p_{k+1}^2$, $1\leq k \leq 1000$. The prime counting functions are 1) $\pi_k$, the number of primes in $s_k$ (gray); 2) $\tilde \pi_k$, the expected number of primes based on the Euler product (red); 3) a version of the latter where only denominators below $p_{k+1}^2$ are included in the expansion of the Euler product (blue). The horizontal lines illustrate the limits tended to by the two probabilistic prime counting functions. 
}
\label{fig:legendreTruncated}
\end{figure}

\begin{figure}
	\centering
	\includegraphics[width=\textwidth]{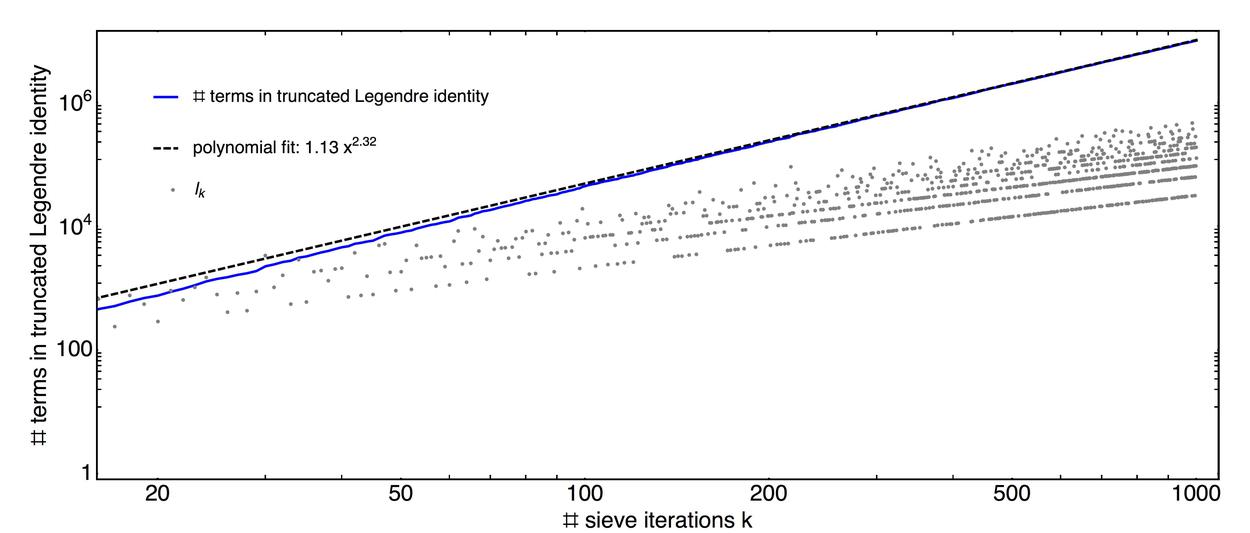}
\caption{
\small \sl
The number of terms remaining in the Legendre identity after truncating every term equal to or larger than $p_{k+1}^2$ plotted as a function of number of sieve steps $k$ (blue). Corresponding to this curve is a best polynomial fit (black, dashed) and for comparison also the interval lengths $l_k$ are shown (dots).
}
\label{fig:legendreError}
\end{figure}

\begin{figure}[H]
\centering
\includegraphics[width=\textwidth]{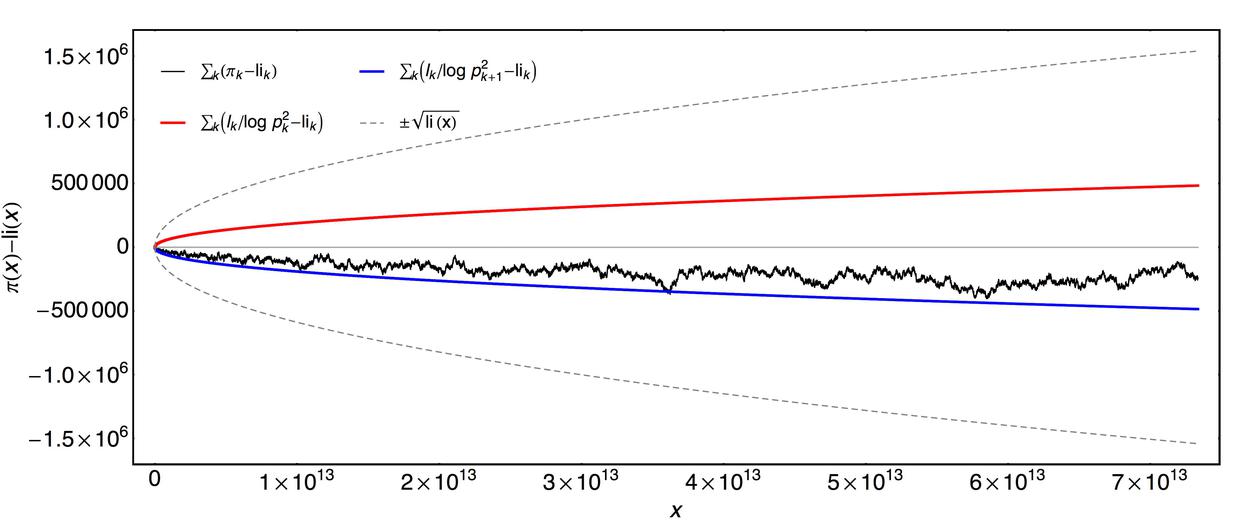}
\caption{
\small \sl
$\pi(x) - \li(x)$ plotted at the values $x = p_{k+1}^2$ for $1\leq k \leq 575200$ (black). Likewise, the differences
$\sum_k (l_k / \log p_k^2 - \li_k)$ (red) and $\sum_k (l_k / \log p_{k+1}^2 - \li_k)$ (blue) are plotted for the same values of $k$.
The gray, dashed lines are the conjectured theoretical bounds for the error term, $\pm \sqrt{\li(x)}$.
}
\label{fig:pix}
\end{figure}

\begin{figure}
\centering
\includegraphics[width=\textwidth]{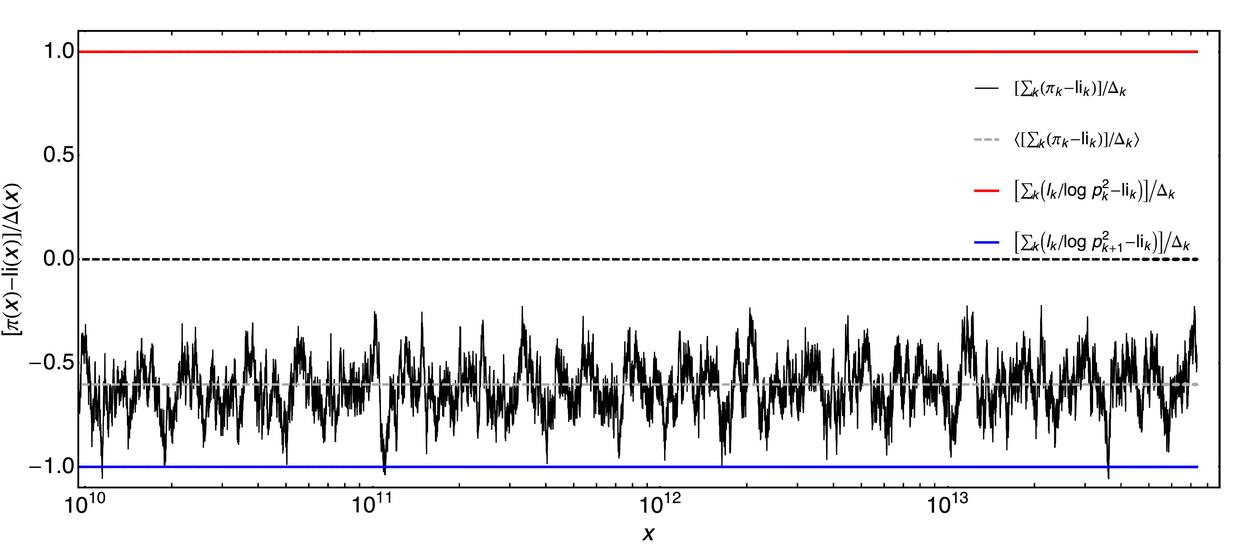}
\caption{
\small \sl
Log-linear plot of $\left[\sum(\pi_k - \li_k)\right]/\Delta_k$, plotted at the values $p_{k+1}^2$ for $1\leq k \leq 575200$ (black). Likewise, the normalized differences 
$\left[\sum_k (l_k / \log p_k^2 - \li_k)\right]/\Delta_k$ (red) and $\left[\sum_k (l_k / \log p_{k+1}^2 - \li_k)\right]/\Delta_k$ (blue) are plotted for the same values of $k$. Also shown is the mean value $\langle \left[\sum(\pi_k - \li_k)\right]/\Delta_k \rangle$ taken across the whole interval (gray, dashed).
}
\label{fig:pixnormalized}
\end{figure}

\begin{figure}
\centering
\includegraphics[width=0.48 \textwidth]{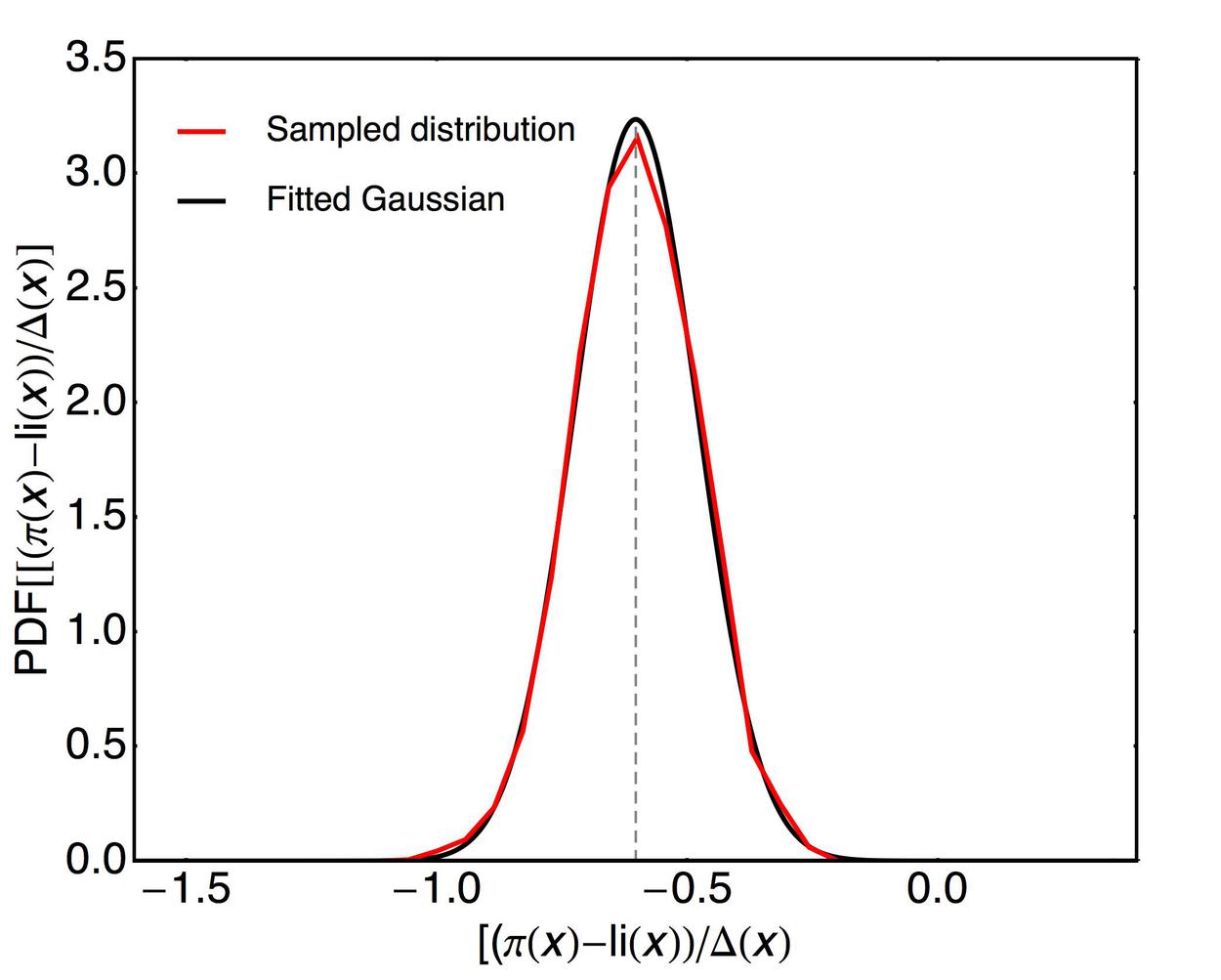}
\caption{
\small \sl
The distribution of the normalized error term $\left[\sum(\pi_k - \li_k)\right]/\Delta_k$ (red), sampled across the values $1\leq k \leq 575200$. The mean and the standard deviation are $\mu= -0.60$, and $\sigma = 0.12$, respectively, both presented with two significant digits. A Gaussian distribution using these parameters is included in the plot (black).
}
\label{fig:pixgauss}
\end{figure}

\begin{figure}
\centering
\includegraphics[width=\textwidth]{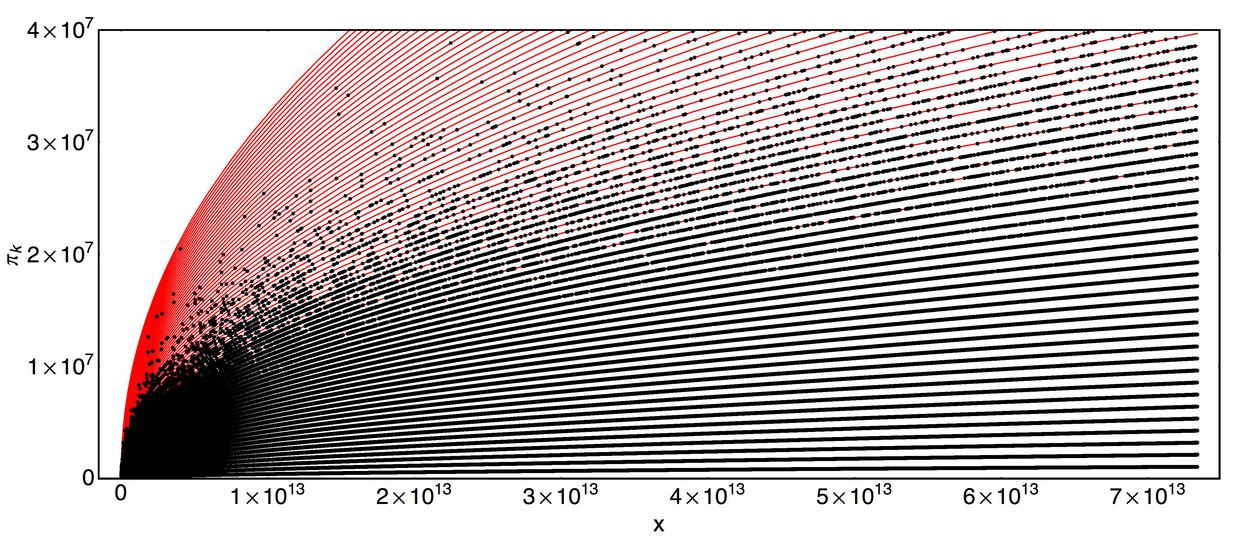}
\caption{
\small \sl
The number of primes $\pi_k$ in each interval $s_k$ plotted at the values $x = p_{k+1}^2$ (black dots). The emerging curves relates to the specific gap values $g_k$; the lower curve corresponds to the set of intervals $s_k$ where $g_k=2$, the second lowest $g_k=4$, etc. The red curves are the theoretical estimates $(2\sqrt{x}g - g)/ \log {x}$, where $g$ takes the gap values $2, 4,\dots$.
}
\label{fig:pik}
\end{figure}

\begin{figure}[H]
	\centering
	\includegraphics[width=\textwidth]{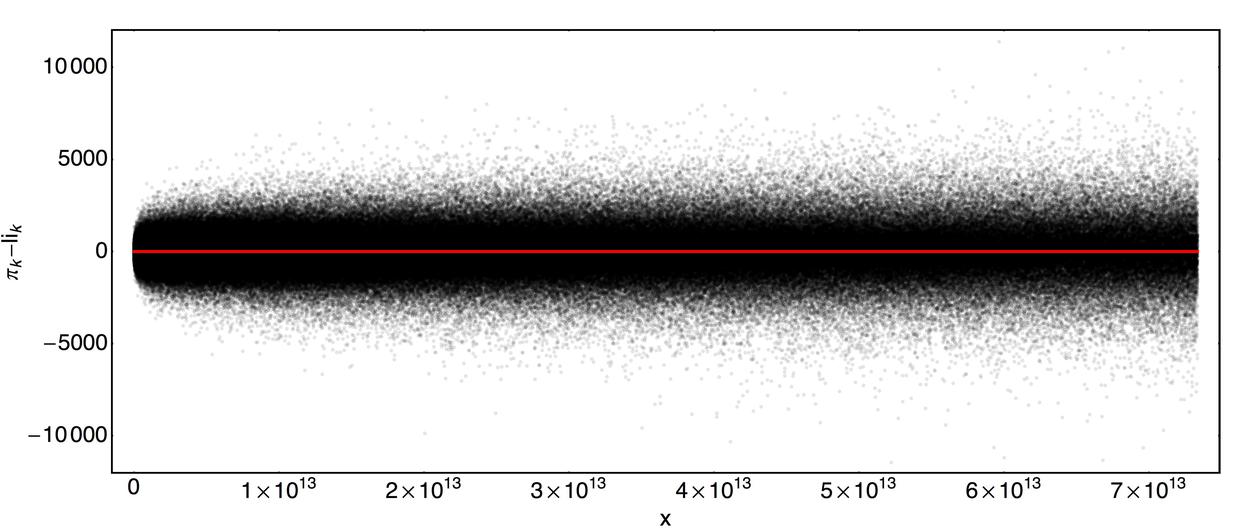}
\caption{
\small \sl
$\pi_k - \li_k$ plotted at the values $x = p_{k+1}^2$, $1\leq k \leq 575200$.
}
\label{fig:diffpiklik}
\end{figure}

\begin{figure}[H]
        \centering
        \begin{subfigure}[b]{0.45\textwidth}
                \centering
                \includegraphics[width=\textwidth]{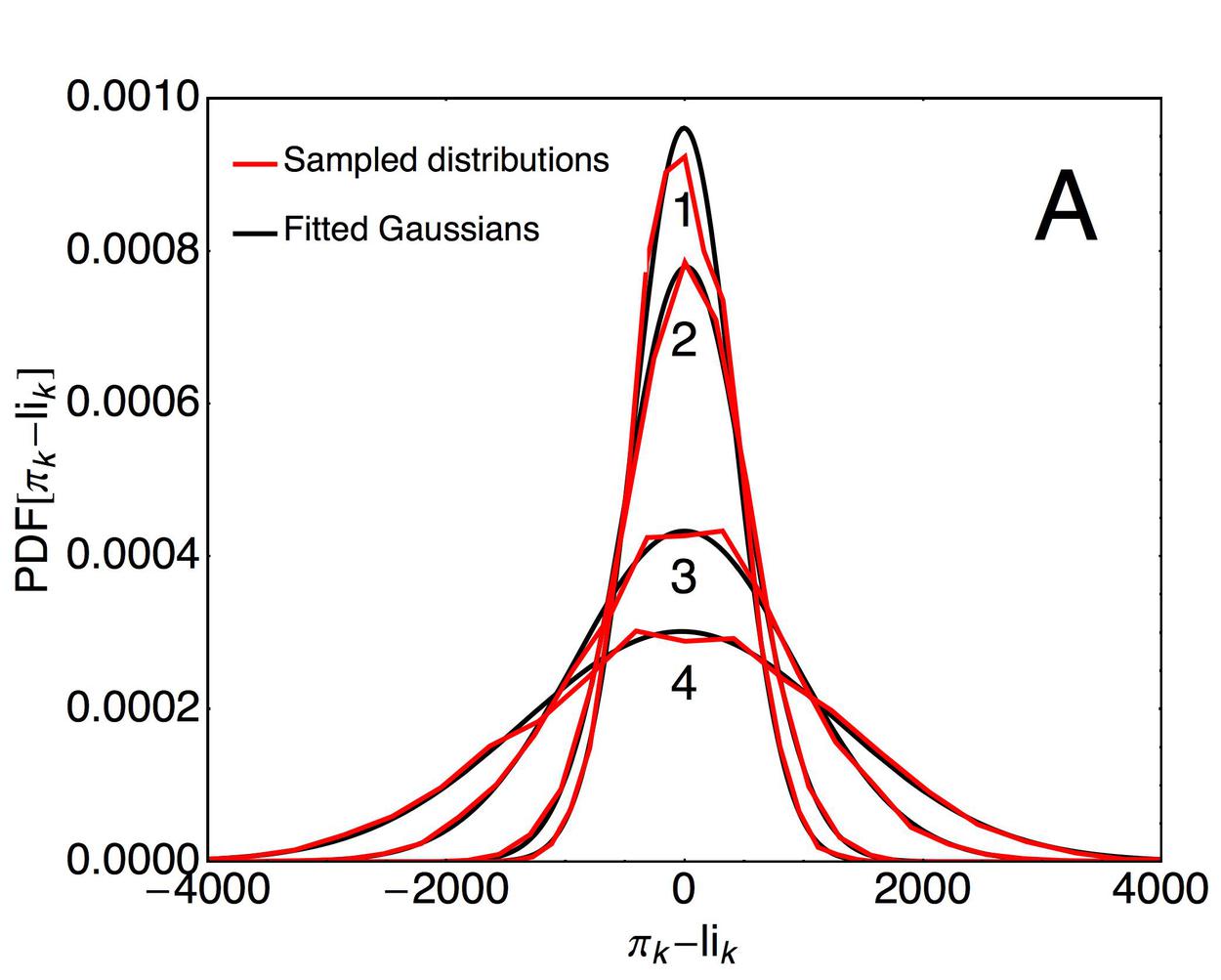}
        \end{subfigure}%
        ~
        \begin{subfigure}[b]{0.45\textwidth}
                \centering
                \includegraphics[width=\textwidth]{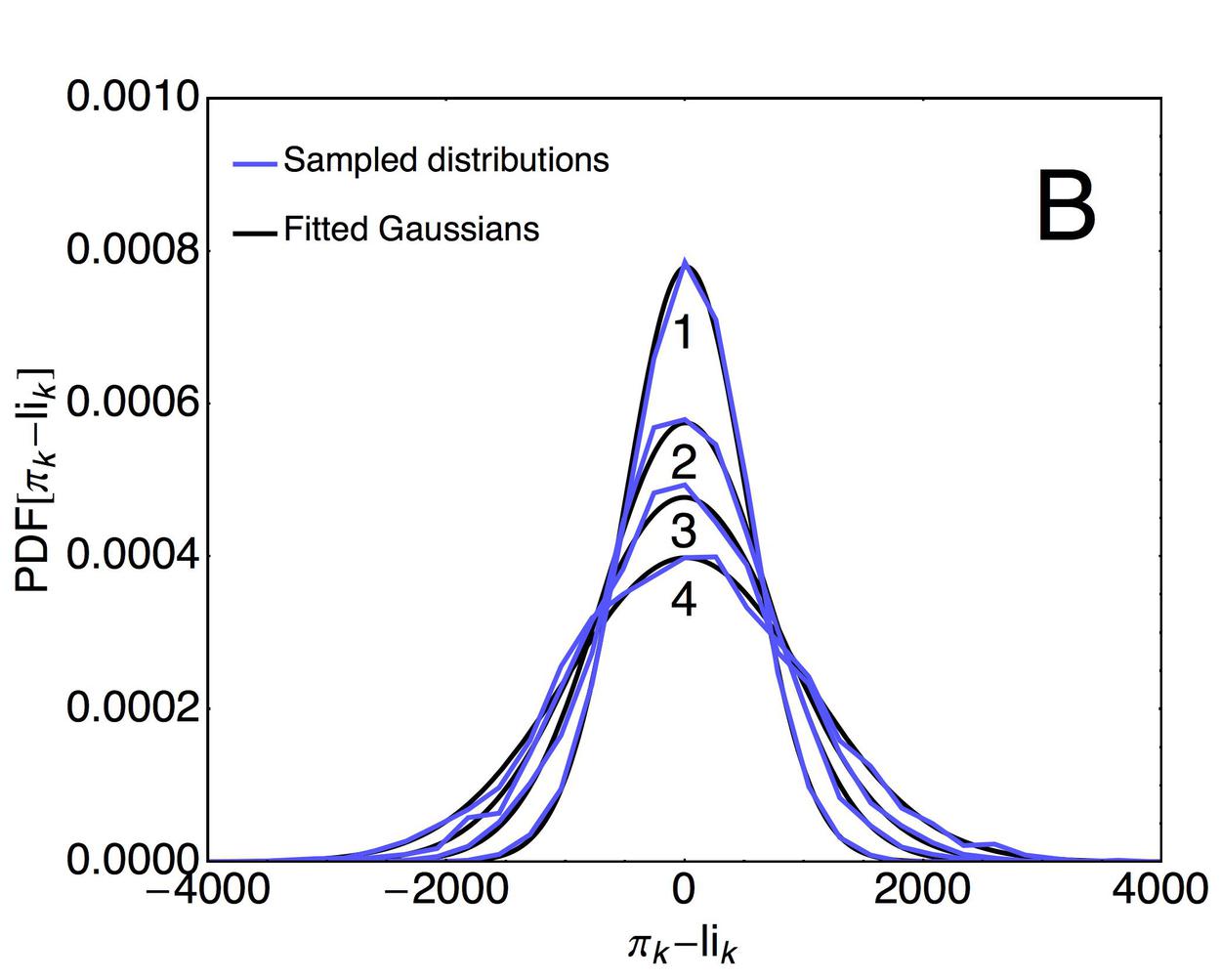}
        \end{subfigure}
\caption{
\small \sl
Empirical and fitted probability distributions for $\pi_k - \li_k$. A) Distributions of $\pi_k - \li_k$ for the sets of intervals $s_k$ where 1) $g_k=2$, 2) $g_k=6$, 3) $g_k=10$, and 4) $g_k=14$. All distributions are build up from sampling $4000$ consecutive values of $\pi_k - \li_k$ for each set of intervals, starting at the $20000$th value in each. B) 
Distributions of $\pi_k - \li_k$ for the set of intervals $s_k$ where $g_k=6$, obtained by sampling 4000 consecutive values starting at value numbers 1) 20000, 2) 40000, 3) 60000, and 4) 80000. In both plots, the black curves are Gaussian distributions with mean and standard deviation identical to the colored curves they fit.
}        
\label{fig:distributionsdiffpiklik}
\end{figure}

\begin{figure}[H]
        \centering
        \begin{subfigure}[b]{0.45\textwidth}
                \centering
                \includegraphics[width=\textwidth]{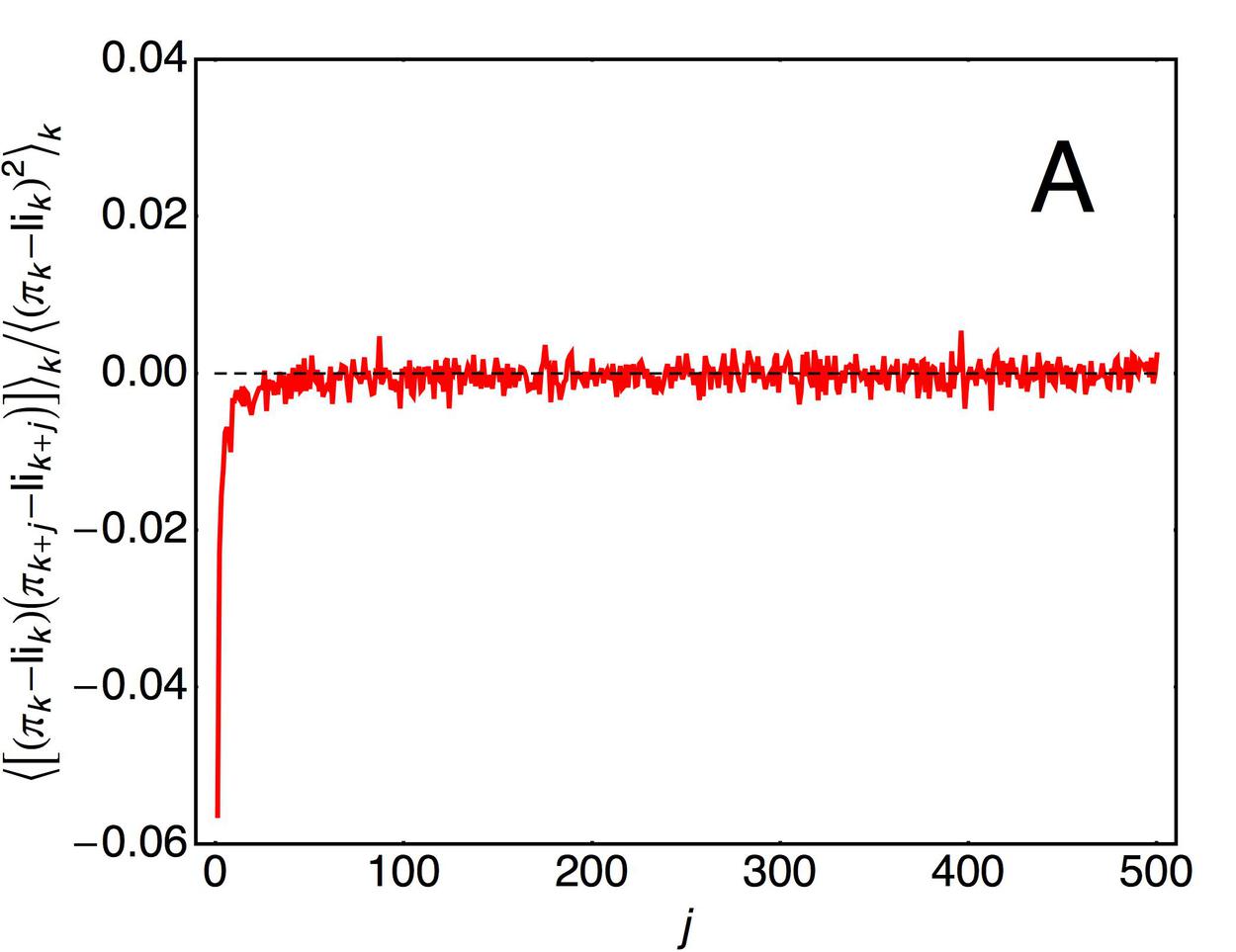}
        \end{subfigure}%
        ~
        \begin{subfigure}[b]{0.45\textwidth}
                \centering
                \includegraphics[width=\textwidth]{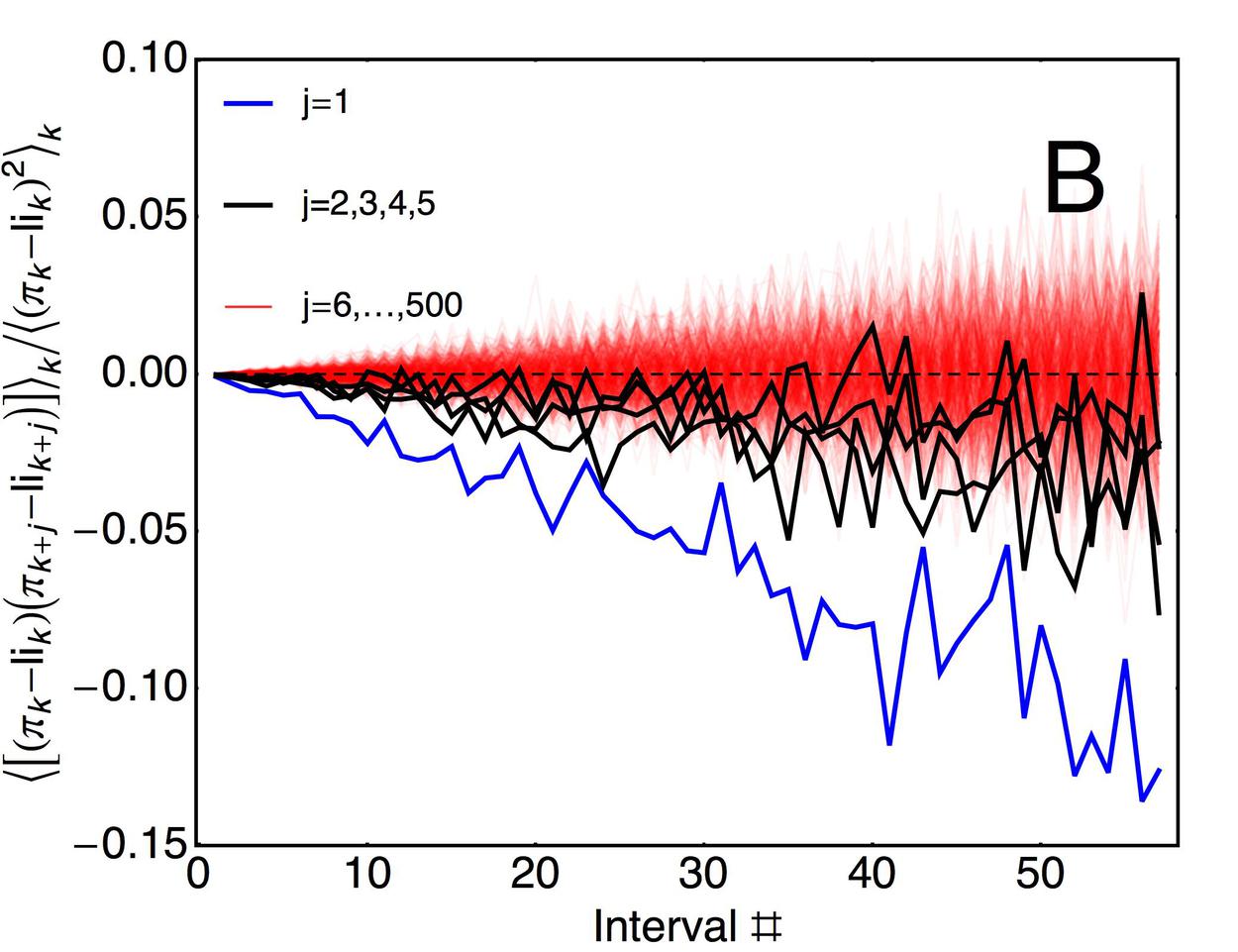}
        \end{subfigure}
\caption{
\small \sl
Measured average correlation between $\pi_k - \li_k$ and $\pi_{k+j} - \li_{k+j}$, plotted for $1\leq j \leq 500$. A) Average is taken over the whole range of $k$, $1\leq k \leq 575200$, and the correlation is plotted as a function of $j$. B) Average is taken over non-overlapping intervals of 10000 values of $k$. The correlation is plotted for each value of $j$ as a function of increasing interval number.  
}        
\label{fig:correlationsdiffpiklik}
\end{figure}

\begin{figure}[H]
        \centering
        \begin{subfigure}[b]{0.45\textwidth}
                \centering
                \includegraphics[width=\textwidth]{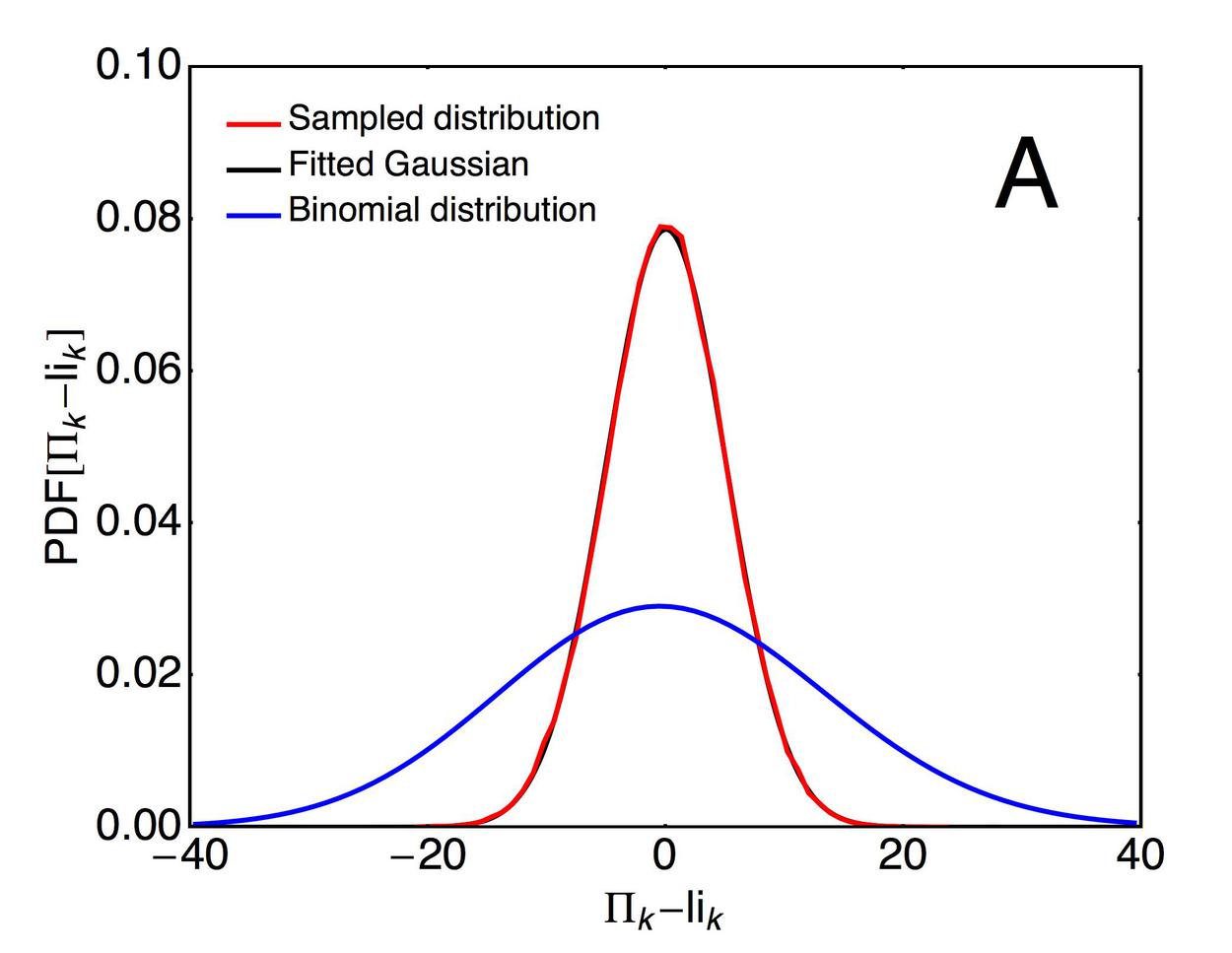}
        \end{subfigure}%
        ~
        \begin{subfigure}[b]{0.45\textwidth}
                \centering
                \includegraphics[width=\textwidth]{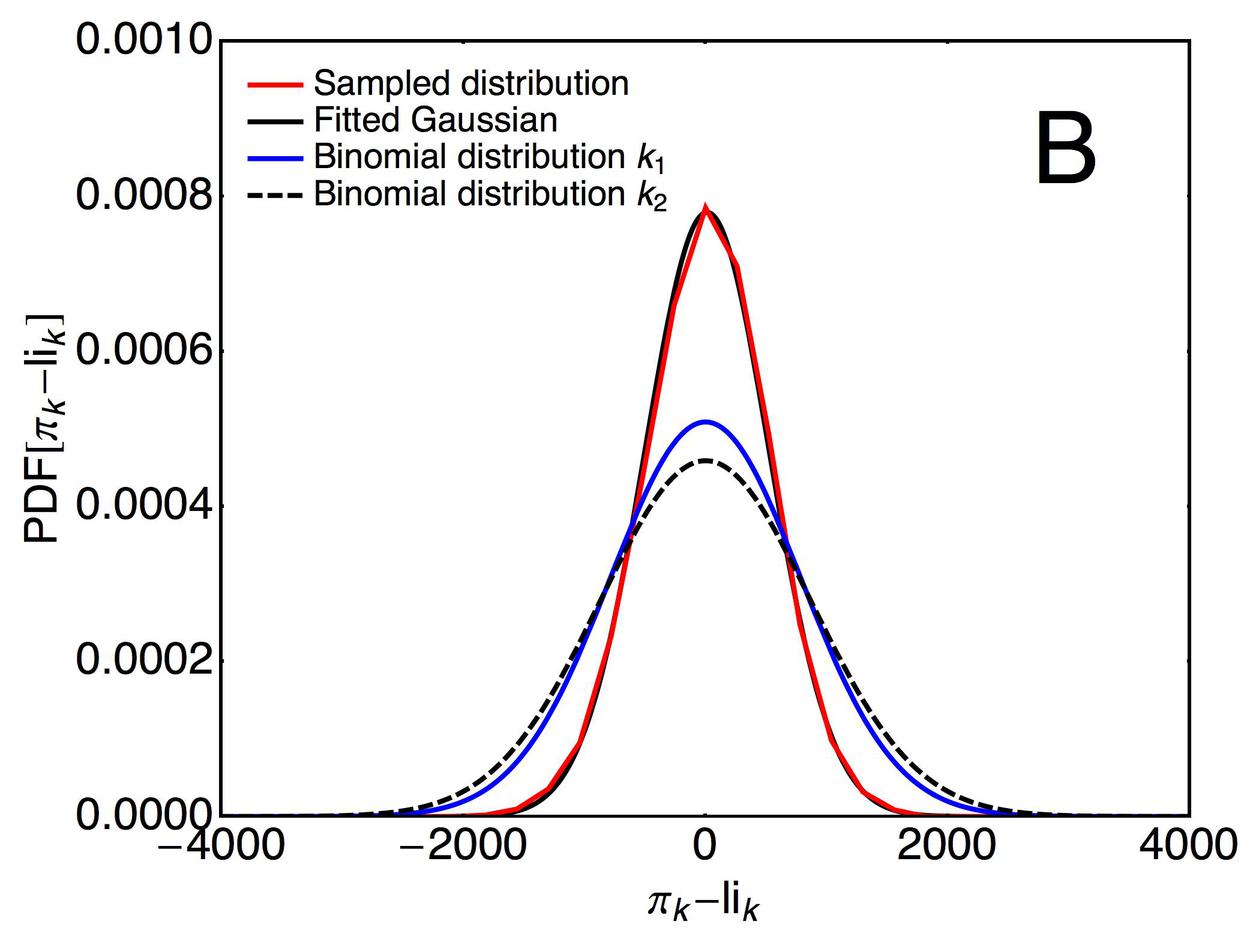}
        \end{subfigure}
\caption{
\small \sl
Empirical (red) and fitted (black) probability distributions of $\pi_k - \li_k$ for A) 100000 samples of the $\pi_k$ random model for the interval $s_{50}$. The blue curve shows the binomial distribution $\textrm{B}(l_{50},l_{50}/\log p_{51}^2)$; B) 4000 consecutive samples of $\pi_k - \li_k$ from the set of intervals $s_k$ with $g_k=6$, starting at sample number 20000.
The blue curve shows the binomial distribution for the interval $s_{k_1}$, corresponding to the first of the 4000 samples. The dashed black curve similarly shows the binomial distribution for the interval $s_{k_2}$, corresponding to the last of the 4000 values.
}        
\label{fig:distbin}
\end{figure}

\begin{figure}
        \centering
                \includegraphics[width=\textwidth]{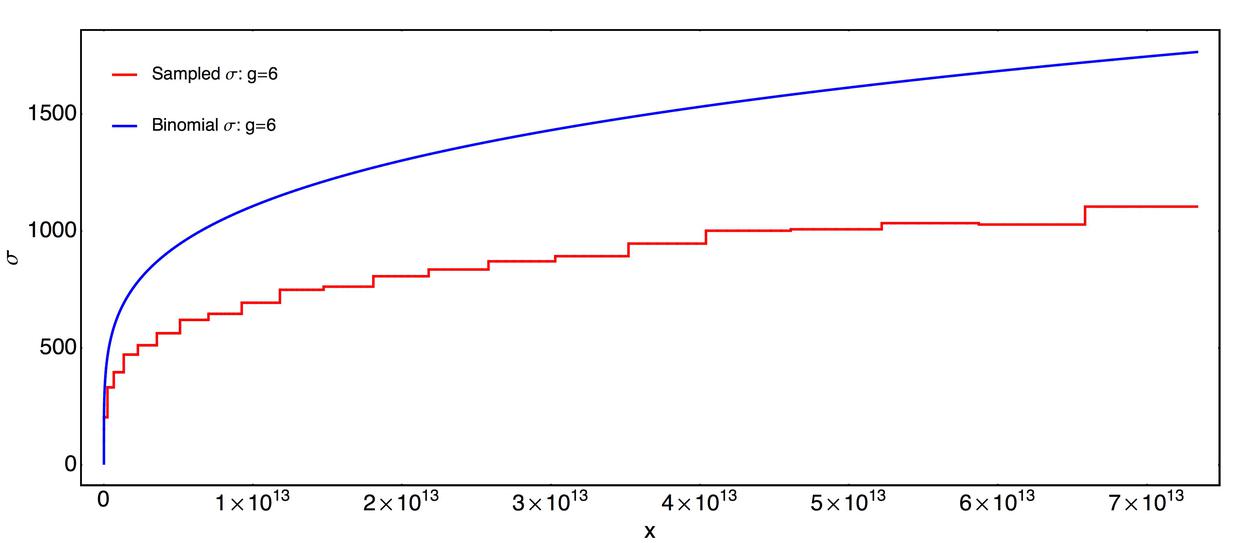}
\caption{
\small \sl
Measured standard deviations of $\pi_k - \li_k$ for the set of intervals $s_k$ where $g_k=6$ (red curve). The horizontal stretches spans 4000 values of $\pi_k - \li_k$ (except the first stretch; spans between 4000 and 8000 values), and the standard deviations are calculated over these. The standard deviation of the binomial distribution for each interval $s_k$ in the same set is shown as the blue curve.
}
\label{fig:sampledstdev}    
\end{figure}

\end{document}